\documentclass[11pt]{amsart}
\usepackage{amsmath}
\usepackage{amsthm}
\usepackage{amssymb}
\usepackage{amsmath}
\usepackage{amsthm}
\usepackage{amssymb}
\usepackage{fancyhdr}
\usepackage{enumerate}
\usepackage{amscd}
\usepackage[all]{xy}
\usepackage{graphicx}
\usepackage{color}

\theoremstyle{plain}

\newtheorem{THM}{Theorem}
\newtheorem{lemma}{Lemma}[section]
\newtheorem{prop}[lemma]{Proposition}
\newtheorem{thm}[lemma]{Theorem}

\newtheorem{defi}[lemma]{Definition}
\newtheorem{cor}[lemma]{Corollary}

\theoremstyle{definition}
\newtheorem*{ack}{Acknowledgements}
\newtheorem{ex}[lemma]{Example}
\newtheorem{remark}[lemma]{Remark}

\theoremstyle{definition}

\newtheorem*{pf2}{Proof of Theorem \ref{thm:PMT}}
\newtheorem*{pfthunbounded}{Proof of Theorem~\ref{th:unbounded}}
\newtheorem*{pfthunbounded-special}{Proof of Theorem~\ref{th:unbounded} (special case)}

\newtheorem*{pfmcf}{Proof of Theorem \ref{th:MCF}}
\newtheorem*{pf-complete}{Proof of Theorem \ref{th:halfspace}}

\numberwithin{equation}{section}

\theoremstyle{remark}

\newtheorem*{pf}{Proof}

\newcommand{\dive}{\textup{div}}

\newcommand{\p}{\partial}

\newenvironment{enumeratei}{\begin{enumerate}[\upshape (i)]}{\end{enumerate}}

\begin{document}
\title[Hypersurfaces with nonnegative scalar curvature]{Hypersurfaces with nonnegative scalar curvature}
\author{Lan--Hsuan Huang}
\address{Department of Mathematics\\
 University of Connecticut\\
 Storrs, CT 06269, USA\\}
\email{lan-hsuan.huang@uconn.edu}

\author{Damin Wu}
\address{Department of Mathematics\\
 University of Connecticut\\
 Storrs, CT 06269, USA\\}

\email{damin.wu@uconn.edu}

\begin{abstract}
We show that closed hypersurfaces in Euclidean space with nonnegative scalar curvature are weakly mean convex. In contrast, the statement is no longer true if the scalar curvature is replaced by the $k$th mean curvature, for $k$ greater than $2$, as we construct the counter-examples for all $k$ greater than $2$. Our proof relies on a new geometric argument which relates the scalar curvature and mean curvature of a hypersurface to the mean curvature of the level sets of a height function. By extending the argument, we show that complete non-compact asymptotically flat hypersurfaces  with nonnegative scalar curvature  are weakly mean convex and prove the positive mass theorem for such hypersurfaces in all dimensions. 
\end{abstract}
\maketitle
\section{Introduction}
For $n$-dimensional hypersurfaces in Euclidean space, it is natural to understand the relations between the intrinsic curvature and the extrinsic curvature. In 1897, Hadamard~\cite{H} proved that a closed (i.e. compact without boundary) surface embedded in $\mathbb{R}^3$ of positive Gaussian curvature is the boundary of a convex body. Hadamard's result was extended by Stoker~\cite{St} to the complete noncompact case. 

In contrast to the strict inequality assumption on the curvature, the non-strict inequality case is more subtle.  
About sixty years later, Chern--Lashof~\cite{CL}  proved that a closed surface in $\mathbb{R}^3$ of non-negative Gauss curvature is the boundary of a weakly convex body.

For $n \ge 2$, Sacksteder~\cite{S} proved that a hypersurface with nonnegative sectional curvature has  semi-positive definite  second fundamental form. His proof used the earlier results of van~Heijenoort~\cite{V} and Hartman--Nirenberg~\cite{HN}. A simpler proof was later provided by do~Carmo--Lima~\cite{DL}. The further study of convex hypersurfaces can be found in, for example, H.~Wu~\cite{HWu} and the references therein.

Among various notions of the intrinsic curvature, the sectional curvature is the strongest (pointwise) curvature condition, while the scalar curvature is the weakest. In this paper, we consider the condition only on the scalar curvature without imposing any condition on the sectional curvature. We would like to know what kind of convexity can be implied by the nonnegative scalar curvature. 

Our another motivation comes from the study of the $k$th mean curvature. For an $n$-dimensional hypersurface, its $k$th mean curvature ($1 \le k \le n$), denoted by $\sigma_k$, is defined to be the $k$th symmetric polynomial of its principal curvatures. It is known that $\sigma_{2k}$ is intrinsically defined, while $\sigma_{2k-1}$ is not, for each $k$~\cite{Reilly:1973}. In particular, $\sigma_1$ is the mean curvature, $2 \sigma_2$ is the scalar curvature, and $\sigma_n$ is the Gauss--Kronecker curvature.

If a closed smooth hypersurface has positive $k$th mean curvature, then its $l$-th mean curvature is positive for each $1 \le l \le k$ (see, for example, \cite[Proof of Proposition 3.3]{HS} and \cite{CNSIII}). In particular, when $k=2$, the result follows from the Gauss equation. That is, a closed hypersurface with positive scalar curvature has positive mean curvature (up to an orientation). It is natural to ask  whether the analogous result holds when one replaces the condition $\sigma_k > 0$ by $\sigma_k \ge 0$. Such statement was claimed in \cite[Proposition 3.3]{HS}, while the proof only works for strict inequalities. It turns out that this statement is not true for all $k \ge 3$, as we construct in Section~\ref{se:example} a family of examples with $\sigma_k \ge 0$ but $\sigma_1 < 0$ somewhere, for $k \ge 3$. These examples are inspired by Chern--Lashof~\cite{CL}. 

In contrast to the counter-examples for all $k\ge 3$,  we prove that the statement holds for $k = 2$. More precisely, we have the following result. 
\begin{THM} \label{thm:R}
Let $n\ge 2$ and $M$ a closed embedded $n$-dimensional  $C^{n+1}$ hypersurface in $\mathbb{R}^{n+1}$. If the scalar curvature of $M$ is nonnegative, then its mean curvature $H$  has a sign, i.e., either $H\ge 0$ or $H\le 0$ everywhere on $M$. 
\end{THM}

Theorem~\ref{thm:R} is implied by the following more general theorem. Denote by $M_+$ a connected component of $\{p \in M: H \ge 0 \mbox{ at } p \}$ that contains a point of positive mean curvature. We say that the mean curvature $H$ \emph{changes signs through} $\Gamma$ if  $\Gamma$ is a connected component of $\partial M_+$ and $\Gamma$ intersects the boundary of a connected component of $M\setminus M_+$\footnote{Let $M\setminus M_+ = \sqcup_{\alpha} U_{\alpha}$ where $U_{\alpha}$ are connected components. Then 
\[
	\partial M_+ = \partial (M\setminus M_+) = \partial (\sqcup_{\alpha} U_{\alpha})  = \mbox{cl} (\cup_{\alpha} \partial U_{\alpha}), 
\]
where $\mbox{cl}(V)$ denotes the closure of a set $V$. If $M\setminus M_+$ has finitely many components, then $ \mbox{cl} (\cup_{\alpha} \partial U_{\alpha}) = \cup_{\alpha} \partial U_{\alpha}$ and $\Gamma$ clearly intersects with some $\partial U_{\alpha}$. If $M\setminus M_+$ has infinitely many components, there may exist a connected component of $\partial M_+$ that does not intersect $\partial U_{\alpha}$ for any $\alpha$ (cf. Remark~\ref{re:boundary-component}).}.

\begin{THM} \label{th:unbounded}
Let $n\ge 2$ and $M$ a complete embedded $n$-dimensional $C^{n+1}$ hypersurface in $\mathbb{R}^{n+1}$ with non-negative scalar curvature. Suppose that the mean curvature $H$ of $M$ changes signs. If $H$ changes signs through $\Gamma$, then $\Gamma$ must be unbounded.
\end{THM}

Let $A$ be the second fundamental form of $M$, let $H$ be the mean curvature of $M$, and let $R$ be the scalar curvature of $M$. Denote by $M_0 = \{ p \in M: A = 0 \mbox{ at } p\}$ the set of geodesic points. Throughout this article, we assume that the hypersurface $M$ is embedded and orientable. 

By the Gauss equation $R = H^2 - |A|^2$, if $R\ge 0$, $H$ can possibly vanish and change signs and if $H=0$ at a point, then $A=0$ at that point. This causes the main analytic difficulty to prove Theorem~\ref{th:unbounded}, as several natural geometric differential equations, including the linearized scalar curvature equation and scalar curvature flow, may be fully degenerate at points of zero mean curvature and cease to be globally elliptic or parabolic. Nevertheless, the set of points where $A=0$, denoted by $M_0$, has more structure  because the connected component of $M_0$ lies in a hyperplane \cite{S} ({\it cf}. Lemma~\ref{le:geodesic}). 

A new ingredient in our proof is that we consider the level sets of the height function defined by the hyperplane containing some subset of $M_0$. We derive a geometric inequality which relates the mean curvature and scalar curvature of $M$ to the mean curvature of the level sets  (Theorem \ref{th:HHR}). Therefore, the geometry of $M$ has some quantitative influence on the geometry of its level sets.  We then carefully investigate the level sets and apply the maximum principles to prove the key results Lemma~\ref{le:Hsign} and Theorem~\ref{th:zero}. Note that in~\cite{HW3} we generalize the geometric inequality to hypersurfaces in a larger class of ambient spaces, including the hyperbolic space and the spheres, and obtain other applications.


As an application of Theorem~\ref{thm:R}, we show that nonnegative scalar curvature is preserved by the mean curvature flow. 

\begin{THM} \label{th:MCF}
Let $n\ge 2$ and $M$ a closed embedded $n$-dimensional $C^{n+1}$ hypersurface in $\mathbb{R}^{n+1}$ with nonnegative scalar curvature. Let $\{ M_t\}$ be a solution to the mean curvature flow with initial hypersurface $M$. Then, the scalar curvature of $M_t$ is strictly positive for all $t>0$.
\end{THM}

 Moreover, by using Theorem~\ref{th:unbounded}, we can provide a simple proof to Sacksteder's theorem for the case of closed hypersurfaces (see Theorem~\ref{th:Sa}). Our argument has further applications. For example, in~\cite{HW3}
 we prove the rigidity results for hypersurfaces  with boundary in the sphere whose scalar curvature is greater or equal to $n(n-1)$, parallel to our previous rigidity results~\cite{HW1} in non-positive space form. (We refer the reader to the excellent survey by Brendle~\cite{Brendle} and the references therein, for the recent rigidity results involving scalar curvature.) 
 
In contrast to the case of closed hypersurfaces, the mean curvature  of a non-closed hypersurface with nonnegative scalar curvature may change signs. For example, consider the $n$-dimensional graph in $\mathbb{R}^{n+1}$ defined by the function $f(x^1,\dots, x^n) = (x^n)^3$. The scalar curvature of the graph is zero, but its mean curvature is strictly positive when $x^n>0$ and strictly negative when $x^n<0$.  Nevertheless, we are able to generalize Theorem \ref{thm:R} to complete non-compact \emph{asymptotically flat} hypersurfaces  (see Definition~\ref{de:waf}). This condition is motivated by general relativity.

\begin{THM} \label{th:halfspace}
Let $n\ge 2$ and $M$ a complete embedded $n$-dimensional $C^{n+1}$ asymptotically flat hypersurface of countably many ends in $\mathbb{R}^{n+1}$ with scalar curvature $R\ge 0$. Then $H$ has a sign, i.e., either $H\ge 0$ or $H\le 0$ on $M$.
\end{THM}

Using Theorem~\ref{th:halfspace}, we prove the Riemannian positive mass theorem for asymptotically flat hypersurfaces for all $n\ge 2$. For three-dimensional asymptotically flat manifolds with nonnegative scalar curvature, the positive mass theorem was proved by Schoen--Yau~\cite{SY79, SY81} and Witten~\cite{W}. The proofs have been generalized to asymptotically flat manifolds of dimension $3 \le n \le 7$ or to spin manifolds of dimension $n\ge 3$. For higher dimensional non-spin manifolds, some approaches  have been announced by Lockhamp~\cite{Lo} and by Schoen~\cite{Schoen}. Recently, Lam~\cite{L} proved the positive mass inequality for graphical asymptotically flat hypersurfaces for all $n\ge 2$, \emph{without} the rigidity result. See Bray~\cite{Bray} for a thorough and up-to-date survey article on Riemannian positive mass theorem. Using Theorem~\ref{th:halfspace} and the geometric inequality (Theorem~\ref{th:HHR}), we generalize Lam's result to non-graphical hypersurfaces and obtain rigidity. 

\begin{THM} \label{thm:PMT}
Let $n\ge 2$ and $M$ a complete embedded $n$-dimensional $C^{n+1}$ asymptotically flat hypersurface  of countably many ends in $\mathbb{R}^{n+1}$ with nonnegative scalar curvature. Then the mass on each end is nonnegative. Moreover, if $M$ is connected and the mass of one end is zero, then $M$ is identical to a hyperplane. 
\end{THM}

Because our definition of \emph{asymptotical flatness} imposes  rather weak decay condition on the induced metric of the ends, the mass (Definition~\ref{de:mass}) may be $+\infty$. In Lemma~\ref{le:mass}, we show that the mass is finite and coincides with the classical definition of the ADM mass if the growth rate of the end is controlled. Our assumptions on the positive mass theorem are rather general and include interesting examples, such as $n$-dimensional Schwarzschild manifolds embedded in $\mathbb{R}^{n+1}$ with two ends (see  Example~\ref{ex:schwarzschild}). We remark that although hypersurfaces in $\mathbb{R}^{n+1}$ are spin, Theorem~\ref{thm:PMT} holds under more general asymptotics and does not seem to be a special case of the positive mass theorem for spin manifolds. 

Note that in~\cite{HW-Penrose}, we extend Theorem~\ref{th:halfspace} to asymptotically flat graphs \emph{with a minimal boundary}, which is a key ingredient to prove the equality case of the Penrose inequality in that setting.

The article is organized as follows. In Section~\ref{se:level}, we prove the geometric inequality (Theorem~\ref{th:HHR}). 
Section \ref{se:thm1} is the most technical part of this article. After establishing several analytical results for the mean curvature operator, we prove Theorem~\ref{th:unbounded}, and apply the results to the mean curvature flow. In addition, we give a shorter proof to Sacksteder's theorem for closed hypersurfaces. 
In Section~\ref{se:example}, we construct the examples of non-mean convex hypersurfaces satisfying $\sigma_k \ge 0$, for all $k\ge 3$. Theorem~\ref{th:halfspace} and Theorem~\ref{thm:PMT} are proven in Section~\ref{se:thm2}. Finally, we include some topological results used in the proof of Theorem~\ref{th:unbounded} in Appendix~\ref{se:appendix}.

\begin{ack}
We thank Brian White for discussions in an early stage of this work, and Huai-Dong Cao, Dan Lee and Xiao Zhang for  valuable comments after the first version appeared on arXiv. The first author thanks Rick Schoen and Mu-Tao Wang for their encouragement and acknowledges NSF grant DMS-$1005560$ and DMS-$1301645$ for partial support. The second author would like to thank The Ohio State University for support.
\end{ack}

\section{The mean curvature of the level sets} \label{se:level}
Let us begin with a linear algebra identity, which applies to a real matrix not necessarily being symmetric.
\begin{prop} \label{pr:id}
  Let $A = (a_{ij})$ be an $n \times n$ matrix with $n \ge 2$ and let $k\in \{ 1, 2, \dots, n\}$. Denote
\[
   \sigma_1(A) = \sum_{i=1}^n a_{ii}, \quad \sigma_1(A|k) = \sigma_1(A) - a_{kk}, \quad \sigma_2(A) = \sum_{1 \le i < j \le n} (a_{ii} a_{jj} - a_{ij}a_{ji}).
\]
  Then, we have
  \begin{equation} \label{eq:id1}
    \begin{split}
    \sigma_1(A) \sigma_1(A|k) 
    & = \sigma_2(A) + \frac{n}{2(n-1)} [\sigma_1(A|k)]^2 + \sum_{1 \le i < j \le n} a_{ij} a_{ji} \\
    & \quad + \frac{1}{2(n-1)} \sum_{1 \le i < j \le n \atop i\neq k, j\neq k } (a_{ii} - a_{jj})^2,
    \end{split}
  \end{equation}
  where the last term is zero when $n = 2$.
  In particular, if $A$ is real and $a_{ij}a_{ji} \ge 0$ for all $1 \le i< j \le n$, then
  \[
     \sigma_1(A) \sigma_1(A|k) \ge \sigma_2(A) + \frac{n}{2(n-1)} [\sigma_1(A|k)]^2
  \]
  with equality if and only if $a_{ii}$ are equal for all $i = 1, \dots, n$ and $i\neq k$, and $a_{ij} a_{ji} = 0$ for all $i,j = 1,\ldots, n$ and $i \ne j$.
\end{prop}
\begin{proof}
  Without loss of generality, we assume $k=1$. Note that
   \[
      \sigma_2(A) = a_{11}\sigma_1(A|1) + \sum_{2 \le i < j \le n} a_{ii} a_{jj} - \sum_{1 \le i <j \le n} a_{ij} a_{ji}. 
   \]
   Then,
   \begin{equation} \label{eq:id2}
     \begin{split}
      \sigma_1(A) \sigma_1(A|1) 
     & = a_{11} \sigma_1(A|1) + \sigma_1(A|1)^2 \\
     & = \sigma_2(A) + \sum_{1\le i <j \le n} a_{ij} a_{ji} + \sigma_1(A|1)^2 - \sum_{2 \le i < j \le n} a_{ii} a_{jj}.
   \end{split}
  \end{equation}
   Now \eqref{eq:id1} follows from applying 
   \[
   		\frac{n-2}{2(n-1)}\Big(\sum_{j=2}^n a_{jj}\Big)^2 - \sum_{2 \le i < j \le n} a_{ii} a_{jj} \\
     = \frac{1}{2(n-1)} \sum_{2 \le i < j \le n} (a_{ii} - a_{jj})^2
   \]
   to the last term on the right hand side of \eqref{eq:id2}.
\end{proof}



We shall adopt the following convention for the mean curvature. Let $N$ be a (piece of) hypersurface in Euclidean space. Let $\mu$ be a unit normal vector field to $N$. The mean curvature of $N$ defined by $\mu$ is given by 
\[
	H_N = - \mbox{div}_0 \mu,
\] 
where $\mbox{div}_0$ is the Euclidean divergence operator. By this convention, the $n$-dimensional sphere of radius $r$ has positive mean curvature $n/r$ with respect to the inward unit normal vector. We denote by $\langle \cdot, \cdot \rangle$ the Euclidean metric on $\mathbb{R}^{n+1}$, and denote by $\p_1,\ldots, \p_{n+1}$ the tangent vectors with respect to $(\mathbb{R}^{n+1}; x^1,\ldots, x^{n+1})$. For a $C^2$ function $f$, we abbreviate $f_i = \p f/\p x^i$, $f_{ij} = \p^2 f/ \p x^i \p x^j$, and denote $Df = (f_1,\ldots, f_n)$.  Let $\eta$ be a vector in $\mathbb{R}^n$. With a slight abuse of notation, we may sometimes view $\eta$ as a vector in $\mathbb{R}^{n+1}$ by letting the $(n+1)$th component be zero.

\begin{thm} \label{th:HHR}
Let $M$ be a $C^2$ hypersurface and let $h: M \to \mathbb{R}$ denote the height function $h(x^1, \dots, x^{n+1})= x^{n+1}$. Assume that $a$ is a regular value of $h$ and  $\Sigma = h^{-1}(a)$ with $| \nabla^M h | > 0$ on $\Sigma$. Denote by $\nu$ and $\eta$ the unit normal vector fields to $M \subset \mathbb{R}^{n+1}$ and $\Sigma \subset \mathbb{R}^n$, respectively; and denote by $H$ and $H_{\Sigma}$ the mean curvatures of $M\subset \mathbb{R}^{n+1}$ and $\Sigma \subset \mathbb{R}^n$ defined by $\nu$ and $\eta$, respectively. Let $R$ be the induced scalar curvature of $M$. Then,
\begin{equation} \label{eq:HHR}
	  \langle \nu, \eta \rangle H H_{\Sigma} \ge \frac{R}{2} + \frac{n}{2(n-1)} \langle \nu, \eta \rangle^2 H_{\Sigma}^2 \qquad \textup{on $\Sigma$}
\end{equation}
with equality at a point in $\Sigma$ if and only if 
$(M, \Sigma)$ satisfies the following two conditions at the point:
\begin{enumeratei}
  \item \label{it:S} $\Sigma \subset \mathbb{R}^n$ is umbilic, with the principal curvature $\kappa$;
  \item \label{it:M} $M \subset \mathbb{R}^{n+1}$ has at most two distinct principal curvatures, and one of them is equal to $\langle \nu, \eta \rangle \kappa$, with multiplicity at least $n-1$.
\end{enumeratei}
\end{thm}
\begin{pf}
It suffices to show \eqref{eq:HHR} at a point $p \in \Sigma$. We may assume $\langle \nu, \eta \rangle \ge 0$ at $p$. Otherwise, we can replace $\eta$ by $-\eta$. 
Let us divide the proof into two cases: 
  
  \textbf{Case 1}: Assume that $\langle \nu, \eta \rangle < 1$ at $p$.  Since $\Sigma = M \cap \{x^{n+1} = a\}$, we have in this case that $\langle \nu, \p_{n+1} \rangle \ne 0$ at $p$. Then, a neighborhood $V$ of $p$ in $M$ can be represented by
   \[
      x^{n+1} = f(x), \qquad \textup{for all $x = (x^1,\ldots, x^n) \in \Omega$},
   \]
   in which $\Omega \subset \{x^{n+1} = 0 \}$ is a small domain containing $p$, and $f \in C^{2}(\Omega)$. It follows that 
   \begin{equation} \label{eq:S1}
      \Sigma \cap V = \{ x \in \Omega \mid f(x) = a \}.
   \end{equation}
 We assume, without loss of generality, that $\langle \nu, \p_{n+1} \rangle > 0$ at $p$; then 
  \begin{equation} \label{eq:nu1}
     \nu = \frac{(-Df, 1)}{\sqrt{1+ |Df|^2} }  \qquad \textup{at $p$}.
  \end{equation}
   We remark that $|Df| \equiv \sqrt{f_1^2 + \cdots + f_n^2} > 0$ at $p$, for, by the construction we have
  \[
      h(x, f(x) ) = f(x) \qquad \textup{for all $x \in \Omega$};
  \]
  thus, $|\nabla^M h| > 0$ on $\Sigma$ implies that $|Df|> 0$ on $\Sigma \cap V$. 
  We can rotate the coordinates $(x^1,\ldots, x^n)$ in $\Omega$ so that at $p$
  \[
      f_1 = |Df|, \qquad \textup{and \quad $f_{\alpha} = 0$ \quad for all $2 \le \alpha \le n$}. 
  \]
  Then, the shape operator $A = (A^i_j)$ on $M \subset \mathbb{R}^{n+1}$ at $p$ is given by
  \[
     A^i_j = \frac{\p}{\p x^j}\Big(\frac{f_i}{w} \Big) = \frac{1}{w} \Big( f_{ij} - \frac{\delta_{i1} f_{1j} |Df|^2}{w^2} \Big),
  \]
  where $w \equiv \sqrt{1 + |Df|^2}$. Hence, at the point $p$,
  \begin{align}
     \sigma_1(A) & = H = \frac{f_{11}}{w^3} + \frac{1}{w}\sum_{\alpha \ge 2} f_{\alpha \alpha}, \notag\\
     \sigma_1(A|1) & =\sum_{\alpha \ge 2} A^{\alpha}_{\alpha} = \frac{1}{w}\sum_{\alpha \ge 2} f_{\alpha \alpha}. \label{eq:A1}
  \end{align}
   On the other hand, by \eqref{eq:S1} and \eqref{eq:nu1} we have $\eta = - Df/|Df|$ at $p \in \Sigma$; hence,
   \[
      \langle \nu, \eta \rangle = \frac{|Df|}{w} > 0 \qquad \textup{at $p$}.
   \]
   Furthermore, the shape operator $A_{\Sigma}$ on $\Sigma \subset \mathbb{R}^n$ is given by
   \begin{equation} \label{eq:AS}
				(A_{\Sigma})^i_j = \frac{\p}{\p x^j} \Big( \frac{f_i}{|Df|}\Big) = \frac{f_{ij}}{|Df|}, \qquad 2 \le i, j \le n.
   \end{equation}
   In particular, the mean curvature
   \begin{equation} \label{eq:HS}
       H_{\Sigma} = \sum_{i=2}^n (A_{\Sigma})^i_i = \frac{1}{|Df|}\sum_{\alpha =2}^n f_{\alpha \alpha}.
   \end{equation}
   Comparing \eqref{eq:A1} and \eqref{eq:HS} we have
   \[
     \sigma_1(A|1) = \frac{|Df|}{w} H_{\Sigma} = \langle \nu, \eta \rangle H_{\Sigma}.
   \]
   Now applying Proposition~\ref{pr:id} with $\sigma_2(A) = R/2$ yields
   \[
      \langle \nu, \eta \rangle H H_{\Sigma}  \ge \frac{R}{2} + \frac{n}{2(n-1)} (\langle \nu, \eta \rangle H_{\Sigma})^2.
   \]
   Here the equality holds if and only if
   \[
      f_{22} = \cdots = f_{nn}, \qquad \textup{and \quad $f_{ij} = 0$ \quad for all $i \ne j$},
   \]
   which, by \eqref{eq:AS}, is the same as that $(M,\Sigma)$ satisfies conditions \eqref{it:S} and \eqref{it:M} at $p$. This proves the result for Case 1.
   
   \textbf{Case 2}: Assume that $\langle \nu, \eta \rangle = 1$ at $p$. Then, $\nu = \eta$ at $p$; equivalently,
   \[
      \langle \nu, \p_{n+1} \rangle = 0.
   \]
   Let us assume, without loss of generality, that $\langle \nu, \p_1 \rangle \ne 0$. We can furthermore rotate the coordinates $(x^1,\ldots, x^n)$ so that 
   \[
      \nu = \p_1 \qquad \textup{at $p$}.
   \]
   Then, by the implicit function theorem, we can represent a neighborhood $U$ of $p$ in $M$ by
   \[
      x^1 = \psi(x^2, \ldots, x^n, x^{n+1}), \qquad \textup{for all $(x^2,\ldots, x^{n+1}) \in \Omega_1$,}
   \]
   where $\Omega_1 \subset \{x^1 = 0\}$ is a small domain containing $p$, and $\psi \in C^{2}(\Omega_1)$ satisfies that
   \begin{equation} \label{eq:dpsi}
      \psi_i(p) = 0 \qquad \textup{for all $2 \le i \le n+1$}.
   \end{equation}
   Since $\Sigma \subset \{x^{n+1} = 0\}$, $\Sigma \cap U$ is given by
   \[
      x^1 = \psi(x', 0) \qquad \textup{for all $(x',0) \equiv (x^2,\ldots, x^{n}, 0) \in \Omega_1$}.
   \]
    By construction above, we have
    \[
       \nu   = \frac{(1, - D' \psi, -\psi_{n+1})}{\sqrt{1 + |D' \psi|^2 + \psi_{n+1}^2}}, \quad \textup{and \quad $\eta  = \frac{(1, - D' \psi, 0)}{\sqrt{1 + |D' \psi|^2}}$},
    \]
    where $D' \psi = (\psi_2,\ldots, \psi_n)$. Using \eqref{eq:dpsi} we obtain the shape operator $A$ for $M \subset \mathbb{R}^{n+1}$ at $p$ 
    \[
       A^i_j = \frac{\p}{\p x^j} \left(\frac{\psi_i}{\sqrt{1 + |D'\psi|^2 + \psi_{n+1}^2 }} \right) = \psi_{ij}, \qquad 2 \le i,j \le n+1,
    \]
    while the shape operator $A_{\Sigma}$ for $\Sigma \subset \mathbb{R}^n$ at $p$ is 
    \begin{equation} \label{eq:AS2}
       (A_{\Sigma})^i_j = \frac{\p}{\p x^{j}}\left(\frac{\psi_i}{\sqrt{1 + |D'\psi|^2 }} \right) = \psi_{ij}, \qquad \textup{for all $2 \le i, j \le n$}.
    \end{equation}
    Hence, at $p$ the matrix $A_{\Sigma}$ is exactly the first $(n-1)\times (n-1)$ principal minor of the matrix $A$. It then follows from Proposition~\ref{pr:id} that
    \[
       H_{\Sigma} H \ge \frac{R}{2} + \frac{n}{2(n-1)} H_{\Sigma}^2,
    \]
    where ``='' holds if and only if 
    \[
       \psi_{22} = \cdots = \psi_{nn}, \qquad \textup{and \quad $\psi_{ij} = 0$ for all $i \ne j$},
    \]
    which is the same as that $(M,\Sigma)$ satisfies conditions \eqref{it:S} and \eqref{it:M} at $p$, by \eqref{eq:AS2}. This proves the result for Case 2. Combining the two cases, we finish the proof. \qed
\end{pf}

\begin{cor} \label{co:HHR}
  With the notations in Theorem~\ref{th:HHR}, if $R\ge 0$ on $M$, then
  \[
     \langle \nu, \eta \rangle H H_{\Sigma} \ge  \frac{n}{2(n-1)} \langle \nu, \eta \rangle^2 H_{\Sigma}^2 \qquad \textup{on $\Sigma$}.
  \]
  In particular, $\langle \nu, \eta \rangle  H H_{\Sigma} \ge 0$ at the point; in addition, if $H= 0$, then $H_{\Sigma} = 0$, and both $M$ and $\Sigma$ are geodesic at the point.
\end{cor}


\section{Complete hypersurfaces with nonnegative scalar curvature} \label{se:thm1}

Let $f$ be a $C^2$ function defined over an open set in $\mathbb{R}^n$. The upward unit normal vector of the graph of $f$ is 
\begin{align} \label{eq:nu}
	\nu = \frac{( -Df, 1) }{\sqrt{1+ |Df|^2}}.
\end{align}
The mean curvature operator is defined by
\[
	H(f) : = - \mbox{div}_0 \nu= \sum_{i,j=1}^n \left( \delta_{ij} - \frac{f_i f_j}{ 1+ |Df|^2} \right) \frac{f_{ij} }{ \sqrt{1+|Df|^2}}.
\]
By this convention, the mean curvature of the lower semi-sphere has positive mean curvature with respect to the upward unit normal vector. 

\begin{prop} \label{pr:nonempty} 
Let $W$ be an open subset in $\mathbb{R}^n$, not necessarily bounded. Let $p\in \p W$, and denote by $B(p)$ an open ball in $\mathbb{R}^n$ centered at $p$. Suppose $f \in C^2(  W\cap B(p)) \cap C^1( \overline{W} \cap B(p))$ satisfies 
\begin{align*}
	& H(f) \ge 0\quad &&\mbox{in } W  \cap B(p)   \\
	&f= c, \; |Df|=0\quad &&\mbox{on } \p W  \cap B(p),
\end{align*}	
for some constant $c$. Then either $f\equiv c$ in $W  \cap B(p)$, or
\[
	\{ x \in  W  \cap B(p): f(x) > c \} \neq \emptyset.
\]	
\end{prop}
\begin{remark}
Notice that we impose no hypothesis on regularity of $\p W$ here and below. In particular, $\partial W$ need not be a hypersurface.
\end{remark}
\begin{pf}
Without loss of generality, we can assume $c=0$. We prove by contradiction. Suppose  $f \le 0$ and $f$ is not identically zero on $W\cap B(p)$. If $f = 0$ at  $x_0 \in W \cap B(p)$, then $x_0$ is a local maximum of $f$. This contradicts the strong maximum principle applied to $H(f)\ge 0$. Thus, $f$ must be strictly negative in $W\cap B(p)$. Because $W\cap B(p)$ is open,  
we can take a smaller open ball $B$ contained in $ W\cap B(p)$ such that $\p B$ touches $\p W \cap B(p)$ at a point $q$; in other words, 
$\partial W$ satisfies an interior sphere condition at $q \in \partial W$. Furthermore, $f < 0$ on $B$ and $f(q) = 0$.  Applying the Hopf boundary lemma to the mean curvature operator $H(f)$ yields that $|Df|(q) \neq 0$, which contradicts with the assumption that $|Df|(q) = 0$ for $q \in \p W \cap B(p)$.
\qed
\end{pf}

\begin{defi} \label{de:convex-point}
Let $W$ be a subset in $\mathbb{R}^n$. A point $p\in \partial W$ is called a convex point of $W$, if there exists an $(n-1)$-dimensional sphere $S$ in $\mathbb{R}^n$ passing through $p$ so that  $\overline{W}\setminus \{ p \}$ is contained in the open ball enclosed by $S$. 
\end{defi}

\begin{remark} \label{re:bounded}
A bounded subset in $\mathbb{R}^n$ has convex points. 
\end{remark}
\begin{lemma} \label{le:Hsign}
Let $W$ be an open subset in $\mathbb{R}^n$. Suppose  $p \in \p W$ is a convex point of $W$. Denote by $B(p)$ an open ball in $\mathbb{R}^n$ centered at $p$. Suppose  $f\in C^2 (W\cap B(p))\cap C^1 ( \overline{W} \cap B(p))$, $f = c, |Df|=0$ on $\p W \cap B(p)$ for some constant $c$. Suppose that almost every real number is a regular value of $f$. If the scalar curvature of the graph of $f$ is nonnegative and $H(f)\ge 0$, then $f\equiv c$ on $\widetilde{B}(p) \cap W$ where $\widetilde{B}(p)$ is an open ball with center $p$ of radius $d_0$ for any $d_0 < \sup \{ d(P, W\cap \partial B(p)):  \mbox{ $P$ is a hyperplane satisfying $P \cap \overline{W} = \{p\}$}\}$, where $d(\cdot, \cdot)$ denotes the Euclidean distance.

As a corollary, if $W$ is bounded,  $f\in C^{n+1}(W\cap N) \cap C^1 ( \overline{W} \cap N)$ for some open set $N$ containing $\partial W$, $f = c, |Df|=0$ on $\p W$, and the graph of $f$ has non-negative scalar curvature and non-negative mean curvature, then $f\equiv c$ in $ \overline{W} \cap N$.

\end{lemma}
\begin{figure}[top] 
   \centering
   \includegraphics[width=0.5\textwidth]{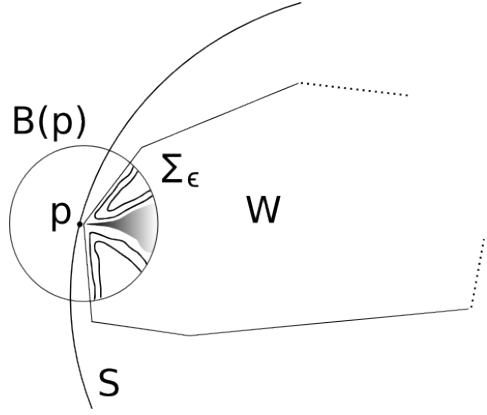}
   \caption{As in the proof of Lemma~\ref{le:Hsign}, $W$ is an open subset in $\mathbb{R}^n$, $p$ is a convex point of $W$, and $S$ is an $(n-1)$-sphere through $p$. For any small ball $B(p)$ centered at $p$, $\{ x \in W \cap B(p): f(x) > c \} \neq \emptyset$, unless $f\equiv c$. Hence the level set $\Sigma_{c+\epsilon}$ of $f$ is nonempty for $\epsilon > 0$ small. The shaded region represents the set $\{x\in W \cap B(p): f<c\}$ (possibly empty). }
\end{figure}
\begin{pf}
By translating the graph if necessary, it suffices to prove the lemma for $c=0$. Suppose to the contrary that $f\not\equiv 0 $ on $\widetilde{B}(p) \cap W$ for any open ball $\widetilde{B}(p)$. Because $p$ is a convex point, there exists an $(n-1)$-dimensional sphere $S$ through $p$ so that $\overline{W} \setminus \{ p \}$ is contained in the open ball enclosed by $S$.  Let the distance between $S$ and $W \cap \partial B(p)$ be $d_0$. Let $\widetilde{B}(p)$ be the open ball centered at $p$ of radius $d_0$. Consider the level sets of $f$ inside $W\cap B(p)$
\[
	\Sigma_{\epsilon} = \{ x \in W \cap B(p): f(x) =\epsilon \}.
\]
By Proposition \ref{pr:nonempty}, $\Sigma_{\epsilon}$ is non-empty in $W\cap \widetilde{B}(p)$ for all $\epsilon>0$ sufficiently small. 
Since almost every real number is a regular value of $f$, $\Sigma_{\epsilon}$ is a $C^2$ hypersurface in $\mathbb{R}^n$  for almost every $\epsilon$. 

We pick a sufficiently small $\epsilon$ so that $|Df|$ does not vanish on $\Sigma_{\epsilon}$, and $\Sigma_{\epsilon}$ intersects $\widetilde{B}(p)$.  Now we continuously translate the $(n-1)$-sphere $S$ toward $\Sigma_{\epsilon}$ along its inward normal at $p$, until it begins to intersect $\Sigma_{\epsilon}$ for the first time. Denote by $S'$ the resulting $(n-1)$-sphere. Then, $S'$ must be tangent to $\Sigma_{\epsilon}$ at an interior point $x_0$ because either $\Sigma_{\epsilon}$ is closed or  $\partial \Sigma_{\epsilon}$ is nonempty and is contained in $\overline{W} \cap \partial B(p)$.  Let $\eta =  Df/ |Df|$ be the normal vector to $\Sigma_{\epsilon}$ in $\mathbb{R}^n$ and let $H_{\Sigma_{\epsilon}}$ be the mean curvature of $\Sigma_{\epsilon}$ with respect to $\eta$. Note that because $f = 0$ on $\p W \cap B(p)$ and $x_0$ is the intersection of $\Sigma_{\epsilon}$ and the $(n-1)$-sphere $S'$ for the first time, $\eta$ at $x_0$ is pointing inward. By comparison principle, $H_{\Sigma_{\epsilon}} > 0$ at  $x_0$.  On the other hand, since the scalar curvature of the graph of $f$ is nonnegative, it follows from Corollary \ref{co:HHR} that
\[
	\langle \nu, \eta \rangle H H_{\Sigma_{\epsilon}} \ge 0.
\]
By \eqref{eq:nu}, $\langle \nu, \eta \rangle < 0$ on $\Sigma_{\epsilon}$. Therefore, $H_{\Sigma_{\epsilon}} \le 0$ at $x_0$. This leads to a contradiction.
\qed
\end{pf}

Let us recall the following characterization of the set of geodesic points, due to Sacksteder~\cite{S}. For the sake of completeness, we include his proof. 
\begin{lemma}\label{le:geodesic} 
Suppose  $M$ is a $C^{n+1}$ hypersurface in $\mathbb{R}^{n+1}$. Denote by $M_0 = \{ p \in M: A = 0 \mbox{ at } p\}$ the set of geodesic points. Let $M_0'$ be a connected component of $M_0$. Then $M_0'$ lies in a hyperplane which is tangent to $M$ at every point in $M_0'$. 
\end{lemma}
\begin{pf}
We consider the Gauss map $\nu: M \rightarrow \mathbb{S}^n$. Since $M$ is of $C^{n+1}$, the Gauss map is of $C^n$. Note that the Gauss map $\nu$ has rank zero at any geodesic point. We can then apply a theorem of Sard~\cite[p. 888, Theorem 6.1]{Sard} to obtain that the image $\nu(M_0)$ is a one-dimensional Hausdorff measure zero set in $\mathbb{S}^n$. It follows that $\nu(M_0)$ is totally disconnected in $\mathbb{S}^n$. Thus, $\nu(M_0')$ consists of a single point in $\mathbb{S}^n$, denoted by $\nu_0$.

It remains to show that $M_0'$ lies in a hyperplane which is orthogonal to $\nu_0$. Pick a point $p_0 \in M_0'$. Let $(V; y^1,\ldots,y^n)$ be a local coordinate chart centered at $p_0$ in $M$, and define
\[
   \varphi(y) = \langle \nu_0, x(y) - x(0) \rangle, \qquad \textup{for each $y = (y^1,\ldots,y^n) \in V$}.
\] 
Here $x = (x^1,\ldots, x^{n+1})$ is the coordinates in $\mathbb{R}^{n+1}$, and $\langle \cdot, \cdot \rangle$ is the Euclidean metric on $\mathbb{R}^{n+1}$.
Then, the function $\varphi \in C^{n+1}(V)$, and by our construction, 
\[
   M_0' \cap V \subset \{ y \in V \mid \frac{\p \varphi}{\p y^i}(y) = 0, \; i = 1,\ldots,n \}.
\]
It follows from a theorem of A. P. Morse~\cite[p. 70, Theorem 4.4]{APM} that $\varphi$ is a constant on $M_0' \cap V$; thus, $\varphi \equiv 0$ on $M_0' \cap V$. Since $M_0'$ is connected, 
\[
   \langle \nu_0, x(p) - x(p_0) \rangle = 0 \qquad \textup{for all $p \in M_0'$},
\] 
namely, $M_0'$ lies in the hyperplane orthogonal to $\nu_0$.
\qed
\end{pf}

The full proof of Theorem~\ref{th:unbounded} is more involved and requires a refinement of Lemma~\ref{le:Hsign}, so we first describe the proof of a \emph{special case} of Theorem~\ref{th:unbounded}, which contains the essential ideas.
\begin{defi} \label{de:separation}
Let $X$ be a topological space and let $E$ be a non-empty closed subset of $X$. We say that $E$ locally separates $X$ if there exists an open neighborhood $N$ of $E$ so that $N \setminus E$ is disconnected. We say that $E$ separates $X$ if $X\setminus E$ is disconnected.
\end{defi}

\begin{pfthunbounded-special}
By Gauss equation, $R\ge 0$ implies that $M_0$ equals the set $\{ p\in M: H=0 \mbox{ at } p\}$. Suppose that $H$ changes signs.  Then  $\{ p\in M: H \ge 0 \mbox{ at } p\}$ has a non-empty boundary. In this proof, we \emph{assume} that the distances between the connected components of $\partial \{ p\in M: H \ge 0 \mbox{ at } p\}$ have a uniform positive lower bound. (For example, this condition is implied if the boundary has only finitely many components.) 

Suppose $H$ changes signs through $\Gamma$. By Lemma~\ref{le:geodesic}, $\Gamma$ lies in a hyperplane $\Pi$ which is tangent to $M$ at $\Gamma$, and $M$ can be represented as the graph of a $C^{n+1}$-function $u$ in an open neighborhood of $\Gamma$ in $\Pi$ with $u=0, |Du|=0$ on $\Gamma$. By the assumption that $\Gamma$ has a uniform positive distance away from other components of $\partial \{ p\in M: H \ge 0 \mbox{ at } p\}$, the subset 
\[
	\Gamma \sqcup \mbox{int}(\{ p \in M: H \ge 0 \mbox{ at } p\}) \sqcup \{ p \in M: H < 0 \mbox{ at } p\}
\]
is an open neighborhood of $\Gamma$ in $M$, where $\mbox{int}(V)$ denotes the interior of a set $V$. Hence, $\Gamma$ locally separates $M$. Because $M$ is graphical near $\Gamma$, $\Gamma$ also locally separates $\Pi$.

Suppose to the contrary that $\Gamma$ is bounded. By Proposition~\ref{pr:enclose} $\Gamma$ encloses a bounded connected open set $W$ in  $\Pi$ with $\p W \subset \Gamma$ and $W\cap \Gamma =\emptyset$. By the assumption that $\Gamma$ has a uniform positive distance away from other components of $\partial \{ p\in M: H \ge 0 \mbox{ at } p\}$, we have either $H\ge 0$ or $H<0$ everywhere in some neighborhood of $\partial W$ in $W$. By Lemma~\ref{le:Hsign}, $u\equiv 0$ on $W$. Applying the argument to all bounded open sets enclosed by $\Gamma$ yields that the mean curvature is zero everywhere on $M_+$.   It contradicts that $M_+$ contains a point of positive mean curvature.
\qed
\end{pfthunbounded-special}

In the above proof, we impose an extra assumption to ensure that $\Gamma$ encloses $W$ and $H$ has a sign in $W$ near $\partial W$. While the assumption is not essential to find a set $W$ enclosed by $\Gamma$, we cannot rule out the possibility that $H$ may change signs in any open subsets of $W$ containing $\partial W$ if the boundary of $\{ p\in M: H \ge 0 \mbox{ at } p\}$ has infinite many components and $\partial W$ contains a limit point of other components. In order to take into account of this, we prove a refinement of Lemma~\ref{le:Hsign} in the following two results. Our goal is to replace the regions of the graph where $H$ changes signs by hyperplanes and to obtain a non-trivial $C^2$ graph. The resulting graph has non-negative scalar curvature and non-negative mean curvature. This would contradict Lemma~\ref{le:Hsign}.

\begin{prop} \label{pr:surgery}
Let $X$ be a contractible open subset of $\mathbb{R}^n$ and let $f\in C^{n+1}(X)$. Suppose the graph of $f$ has non-negative scalar curvature. Let $\Omega \subset X$ be a non-empty connected component of $\{x \in X: H >0 \mbox{ at } (x,f(x)) \}$.  Then there exists $\tilde{f} \in C^2 (X)$ so that $\tilde{f} = f$ on $\Omega$, $\tilde{f}$ is an affine function on each connected component of $X \setminus \Omega$, and the graph of $\tilde{f}$ has non-negative scalar curvature and non-negative mean curvature. Also, almost every real number is a regular value of $\tilde{f}$.
\end{prop}
\begin{proof}
Denote ${\partial \Omega}\cap X =\sqcup_{\alpha} \Sigma_{\alpha} $ where each $\Sigma_{\alpha}$ is a connected component. Note that the graph of $f$ has zero mean curvature on $\Sigma_{\alpha}$. Denote by $\mbox{Graph}[f]$ the graph of $f$. Because the scalar curvature is non-negative, by Lemma~\ref{le:geodesic} $\mbox{Graph}[f]\big|_{\Sigma_{\alpha}}$ is contained in a hyperplane $\Pi_{\alpha} \subset \mathbb{R}^{n+1}$ for each $\alpha$ and $\textup{Graph}[f]$ is tangent to $\Pi_{\alpha}$ with $|D^2f|=0$ on $\Sigma_{\alpha}$.

Let $\tilde{f} = f$ on $\overline{\Omega}\cap X$. By Proposition~\ref{pr:boundary}, the boundary of each connected component of $X\setminus \overline{\Omega}$ is connected, so we can define $\tilde{f}$ across $\Sigma_{\alpha}$ on each connected component of $X \setminus \overline{\Omega}$ by the affine function that defines $\Pi_{\alpha}$. Because $|D^2f|=0 = | D^2 \tilde{f} |$ on $\Sigma_{\alpha}$, $\tilde{f} \in C^2 (X)$. By Sard's theorem, almost every real number is the regular value of $f$. By construction, the regular value of $f$ is also the regular value of $\tilde{f}$. 
\end{proof}

\begin{thm} \label{th:zero} 
Let $W$ be an open subset in $\mathbb{R}^n$  and let $p \in \partial W$ be a convex point of $W$. Denote by $B(p)$ an open ball in $\mathbb{R}^n$ centered at $p$. Suppose $f\in C^{n+1}(\overline{W}\cap B(p))$ and $f=c, |Df|=0, |D^2 f|=0$ on $\partial W \cap B(p)$ for some constant $c$. If the scalar curvature of the graph of $f$ is nonnegative, then $f\equiv c$ in $\widetilde{B}(p)\cap W$ for  an open ball $\widetilde{B}(p)$ centered at $p$ of radius $d_0$ for any $0< d_0 < \sup \{ d(P, W\cap \partial B(p)):  \mbox{ $P$ is a hyperplane with $P \cap \overline{W} = \{p\}$}\}$, where $d(\cdot, \cdot)$ is the Euclidean distance.

As a corollary, if $W$ is bounded and $f\in C^{n+1}(\overline{W}\cap N)$ for some open set $N$ containing $\partial W$, $f=c, |Df|=0, |D^2 f|=0$ on $\partial W$, and the graph of $f$ has non-negative scalar curvature. Then $f\equiv c$ on $\overline{W}\cap N$.
\end{thm}
\begin{proof}
 Without loss of generality we assume $c=0$. Suppose  to the contrary $f\not\equiv 0$ on $\widetilde{B}(p) \cap W$.  By Lemma~\ref{le:Hsign}, $H$ changes signs in $\widetilde{B} \cap W$. Choose $x_0\in W\cap \widetilde{B}(p)$ so that $H>0$ at $(x_0, f(x_0))$.  Let $\Omega \subset W\cap B(p)$ denote the connected component of $\{x\in W \cap B(p): H > 0 \mbox{ at } (x,f(x))\}$ containing $x_0$. 

By Proposition~\ref{pr:surgery} (with $X= B(p)$), there exists $\tilde{f} \in C^2(B(p))$ so that $\tilde{f} = f$ on $\Omega$,  $\tilde{f}$ is an affine function on each connected component of $ B(p)\setminus \overline{\Omega}$, and  the graph of $\tilde{f}$ has non-negative scalar curvature and non-negative mean curvature.  Because $ |D^2f|=0$ on $\partial W\cap B(p)$, the mean curvature  of the graph of $f$ is zero on $\partial W\cap B(p)$ and hence $\partial W\cap B(p)$ lies in $B(p)\setminus\Omega$. Therefore, $\tilde{f}$ is an affine function on $\partial W\cap B(p)$.

If $\tilde{f} = 0= |D \tilde{f}|$ on $\partial W \cap B(p)$, then Lemma~\ref{le:Hsign} implies that $\tilde{f} \equiv 0$ in $\widetilde{B}(p)\cap W$. It contradicts that $H>0$ at $(x_0, f(x_0))$.

If $\tilde{f} $ is another affine function on $\partial W \cap B(p)$, we first translate $\tilde{f}$ so that $\tilde{f} = 0$ at $p$ and then rotate the graph of $\tilde{f}$, still denote the graphing function by $\tilde{f}$,  so that $|D\tilde{f}|=0$ and $\tilde{f} = 0$ on $\partial W' \cap B'(p)$ for some open set  $W'$ and an open ball $B'(p)$.  Again applying Lemma~\ref{le:Hsign}  leads a contradiction.
\end{proof}

The following maximum principle type result will be used in Section~\ref{se:thm2}.
\begin{lemma} \label{le:maximum-principle}
Denote by $B_r$ the open ball in $\mathbb{R}^n$ centered at the origin of radius $r$. Let $f\in C^n (B_{r_2}\setminus \overline{B}_{r_1})\cap C^1 (\overline{B_{r_2}\setminus B_{r_1}})$ for some $r_2 > r_1 > 0$. Suppose that $f$ satisfies $H(f) \ge 0$ and the scalar curvature of the graph of $f$ is nonnegative. Then 
\[
	 \max_{\overline{B}_{r_2} \setminus B_{r_1}} f =\max_{\p B_{r_2}} f . 
\]
Moreover, if $\displaystyle f(x) = \max_{\p B_{r_2}} f $ for some interior point $x \in B_{r_2}\setminus \overline{B}_{r_1}$, then $f \equiv  \max_{\p B_{r_2}} f$ in $\overline{B_{r_2}\setminus B_{r_1}}$.
\end{lemma}
\begin{remark}
Lemma~\ref{le:maximum-principle} does not follow directly from applying the standard maximum principle to $H(f) \ge 0$, since we impose no hypothesis on $\max_{\p B_{r_1}} f$. This result may have interests of its own.
\end{remark}
\begin{pf}
By subtracting $\max_{\p B_{r_2}} f$ from $f$,  we may assume $\max_{\p B_{r_2}} f  =0$. Suppose to the contrary that $f$ is not identically zero and $f>0$ somewhere in $B_{r_2} \setminus B_{r_1}$. Then the level set
\[
	\Sigma_{\epsilon} = \{ x \in B_{r_2} \setminus B_{r_1}: f(x) =\epsilon \}
\]
is non-empty for $\epsilon>0$ sufficiently small.   By Morse--Sard theorem, for almost every small $\epsilon$, $\Sigma_{\epsilon}$ is a piece of $C^{n}$ hypersurface in $\mathbb{R}^n$. Note that either $\Sigma_{\epsilon}$ has no boundary, or $\partial \Sigma_{\epsilon}$ is contained in $\p B_{r_1}$. Fix $\epsilon>0$ so that $|Df|$ does not vanish on $\Sigma_{\epsilon}$. Let $p \in \p B_{r_2}$ be a point that is closest to $\Sigma_{\epsilon}$. Now we continuously translate the $(n-1)$-sphere $\p B_{r_2}$ toward $\Sigma_{\epsilon}$, along its inward normal at $p$. Denote by $S'$ the $(n-1)$-sphere that touches $\Sigma_{\epsilon}$ for the first time. Then, $S'$ must be tangent to $\Sigma_{\epsilon}$ at an interior point $x_0$, and $\Sigma_{\epsilon}$ lies in the ball enclosed by $S'$.  Then by comparison principle, the mean curvature $H_{\Sigma_{\epsilon}}$ of $\Sigma_{\epsilon}$ with respect to the inward unit normal vector $Df/ |Df|$ is positive at $x_0$. However, $H_{\Sigma_{\epsilon}} \le 0$ at $x_0$  by Corollary~\ref{co:HHR}, and it leads a contradiction.

Last, if $f(x) =  \max_{\p B_{r_2}} f $ for some interior point $x$, then by strong maximum principle and $H(f) \ge 0$ we prove that $f \equiv  \max_{\p B_{r_2}} f$ in $\overline{B_{r_2}\setminus B_{r_1}}$. 
\qed
\end{pf}


\begin{pfthunbounded}
Recall that $M_+$ is a connected component of $\{p \in M: H \ge 0 \mbox{ at } p\}$ that contains a point of positive mean curvature and $\Gamma$ is a connected component of $\partial  M_+$ that intersects the boundary of a connected component of $M\setminus M_+$. Because $\Gamma \subset M_0$, by Lemma~\ref{le:geodesic} $\Gamma$ lies in a hyperplane $\Pi$ and $M$ is locally the graph of a $C^{n+1}$ function $f$ over an open neighborhood of $\Gamma$ in $\Pi$. 

Suppose to the contrary that $\Gamma$ is a bounded. Then $\Gamma$ is contained in a compact subset $K$ of $\Pi$. Let $\Omega$ be the connected component of $\{x \in K \subset \Pi:  \mbox{$M$ is graphical and } H \ge 0 \mbox{ at } (x, f(x))\}$ that contains $\Gamma$. By choosing $K$ sufficiently large,  $\Omega$ contains a point of positive mean curvature.  Let $\Pi \setminus \Omega = \sqcup_{\alpha } U_{\alpha}$ where each $U_{\alpha}$ is a connected component. By Proposition~\ref{pr:boundary}, $\partial U_{\alpha}$ is connected. Because $\Gamma$ intersects the boundary of a connected component of $M\setminus H_+$, $\Gamma$ contains $\partial U_{\alpha_0}$ for some $\alpha_0$. First, note that $U_{\alpha_0}$ must be unbounded. Otherwise, Theorem~\ref{th:zero} implies  $H\equiv 0$ on $U_{\alpha_0}$, which contradicts that $U_{\alpha_0}$ is in the complement of $\Omega$. Hence, because $\Omega$ is bounded, $U_{\alpha_0}$ is the unique unbounded component containing infinity. Therefore, $\Pi \setminus U_{\alpha_0}$ is bounded with the boundary contained in $\Gamma$. By Theorem~\ref{th:zero} again, it contradicts that $\Omega$ contains a point of positive mean curvature.

\qed
\end{pfthunbounded}
Using Theorem~\ref{th:unbounded}, we provide another proof to Sacksteder's theorem for closed hypersurfaces. 
\begin{thm} \label{th:Sa}
Suppose $M$ is a $C^{n+1}$-smooth closed hypersurface in $\mathbb{R}^{n+1}$. If the sectional curvature of $M$ is nonnegative, then the induced second fundamental form is semi-positive definite. As a consequence, $M$ is the boundary of a convex body in $\mathbb{R}^{n+1}$.
\end{thm}
\begin{pf}
Let $(A_{ij})$ be the second fundamental form of $M$. Because the sectional curvature of $M$ is nonnegative, at a point in $M$, the principal curvatures are either all nonnegative or all non-positive,  Suppose to the contrary that $(A_{ij})$ is not semi-positive definite. Then by taking the traces of the sectional curvature and $(A_{ij})$, $M$ has non-negative scalar curvature and its mean curvature changes signs. It contradicts Theorem~\ref{thm:R}.
\qed
\end{pf}

Theorem~\ref{thm:R} has applications in the mean curvature flow when the initial hypersurface has nonnegative scalar curvature. Let us briefly recall the setting of the mean curvature flow. Let $M$ be a closed hypersurface in $\mathbb{R}^{n+1}$ represented by a diffeomorphism: For  an open subset $U \subset \mathbb{R}^n$, 
\[
	F_0 : U \rightarrow F_0(U) \subset M \subset \mathbb{R}^{n+1}.
\]
Let $F(x,t)$ be a family of maps satisfying
\begin{align*} 
	& \frac{\partial}{\partial t} F(x,t) = H(x,t)\nu(x,t), \quad x \in U,\\
	& F(\cdot, 0) = F_0,\notag
\end{align*}
where $\nu(\cdot,t)$ is the inward unit normal vector to $M_t:= F(M, t)$ and $H(\cdot, t)$ is the mean curvature with respect to $\nu$. The family of closed hypersurfaces $\{ M_t\}$ for $t > 0$ is called a solution to the mean curvature flow.

\begin{pfmcf}
By Theorem \ref{thm:R}, the mean curvature of $M$ is nonnegative. It is well known that 
if $M$ has nonnegative mean curvature, then $M_t$ has positive mean curvature for all $t>0$. We then consider
\[
	q_2 : = \frac{R}{2H},
\]
where $H$ and $R$ are the mean curvature and scalar curvature of $M_t$, respectively.
A result of Huisken--Sinestrari~\cite[p. 61, Corollary 3.2]{HS} shows that the evolution equation of $q_2$ satisfies the parabolic  strong  maximum principle. It follows that $q_2 > 0$ on $M_t$ for all $t>0$, because $q_2 \ge 0$ on $M$. Thus, we conclude that $R>0$ on $M_t$ for all $t>0$. 
\qed
\end{pfmcf}


\section{Examples of nonnegative $k$th mean curvature} \label{se:example}
Let $M$ be a smooth closed hypersurface in $\mathbb{R}^{n+1}$. Denote by $\kappa_i$, $i=1,\ldots,n$, the principal curvatures of $M$. We define, for each $1 \le k \le n$, the \emph{$k$th mean curvature} of $M$ to be
\[
   \sigma_k(A) = \sum_{1 \le i_1 < \cdots < i_k \le n} \kappa_{i_1}\cdots \kappa_{i_k}.
\]
In particular, $\sigma_1(A)$, $2 \sigma_2(A)$, and $\sigma_n(A)$ are the mean curvature, the scalar curvature, and the Gauss--Kronecker curvature of $M$, respectively.

It is well known that if $\sigma_k(A) > 0$ then $\sigma_l(A) > 0$ for all $1 \le l \le k$ (see, for example, \cite{CNSIII} and \cite[p. 51]{HS}). We are interested in the non-strict inequality case: Namely, the question is, \emph{whether $\sigma_k(A) \ge 0$ would imply $\sigma_l(A) \ge 0$ for all $1 \le l \le k$}? 

For $k = 2$, Theorem~\ref{thm:R} tells us that $\sigma_2(A) \ge 0$ implies $\sigma_1(A) \ge 0$. However, it is no longer true for $k \ge 3$.
In fact, for any $n \ge 3$ and $k \ge 3$, we construct below a family of smooth closed hypersurfaces, which satisfy $\sigma_k(A) \ge 0$ but are not mean convex, i.e., $\sigma_1(A)$ changes signs.

\begin{ex} \label{ex:odd}
  Let $n \ge 3$ and $k$ be an odd integer such that $3 \le k \le n$. Consider the hypersurface in $\mathbb{R}^{n+1}$ given by
  \begin{equation} \label{eq:odd}
     (r - a)^2 + (x^{n+1})^2 = 1,
  \end{equation}
  where $a > 1$ is a constant, and
  \[
     r = \sqrt{(x^1)^2 + \cdots + (x^n)^2}.
  \]
  This hypersurface is homeomorphic to $\mathbb{S}^1 \times \mathbb{S}^{n-1}$, for it can be obtained by rotating a unit circle about the $x^{n+1}$-axis.

 We shall show that, for each $a \in (n/k, n)$,  $\sigma_k(A) \ge 0$ but $\sigma_1(A)$ change signs. (When $k<n$, the same conclusion holds if $a = n/k$.)
In particular, letting $a = n/2$, the hypersurface given by \eqref{eq:odd} has nonnegative $k$th mean curvature while its first mean curvature changes signs. 
  
   Note that the hypersurface is symmetric about $\{ x^{n+1} = 0\}$. Let us consider the lower half portion $x^{n+1} = \phi(r)$,
  where 
  \[
     \phi (r) = -[ 1 - (r -a)^2 ]^{1/2}, \qquad \textup{for all $a-1 \le r \le a+1$}.
  \]
By a direct computation, the $k$th mean curvature of the graph of $x^{n+1} = \phi(r)$ is 
\begin{align*}
   \sigma_k(A) 
   & = \binom{n-1}{k-1} \frac{(\phi')^{k-1}}{r^{k-1}[ 1 + (\phi')^2]^{k/2}} \left(\frac{\phi''}{1 + (\phi')^2} + \frac{n - k}{k} \frac{\phi'}{r} \right) \\
   & = \binom{n-1}{k-1} \big(1 - \frac{a}{r}\big)^{k-1} \left[1 + \frac{n-k}{k}\big(1 - \frac{a}{r}\big) \right].
\end{align*}
Since $k$ is odd, to get $\sigma_k(A)\ge 0$ it suffices to consider 
\begin{equation*} 
   0 \le 1 + \frac{n-k}{k} \big( 1 - \frac{a}{r}\big).
\end{equation*}
That is, when $1 \le k < n$, $\sigma_k(A) \ge 0$ if
\begin{equation*} 
    1 - \frac{a}{r} \ge - \frac{k}{n - k}, \qquad \mbox{ for all } a -1 < r < a+1;
\end{equation*}
and when $k =n$ (then $n$ is odd), $\sigma_n(A)$ is always nonnegative.
On the other hand, 
\[
	\frac{1}{a + 1} > 1 - \frac{a}{r} >  -\frac{1}{a-1}, \qquad \mbox{for $a -1 < r < a+1$.}
\]
Therefore, for all the real numbers $a$ satisfying
\[
   -\frac{1}{n-1} >  - \frac{1}{a - 1} > - \frac{k}{n - k}, \qquad \textup{i.e.}, \quad n > a > \frac{n}{k},
\]
we have $\sigma_k (A) \ge 0$, but $\sigma_1(A)$ changes signs. More precisely, let us fix any real number $a \in (n/k, n)$; then
\begin{align*}
	\sigma_1(A) & > 0, \quad \mbox{ for }  \frac{n-1}{n}a < r <  a + 1, \quad \textup{ and}\\
	\sigma_1(A) & < 0, \quad \mbox{ for }   a - 1 < r < \frac{n-1}{n}a.
\end{align*}
\end{ex}

\begin{ex} \label{ex:even}
Let $n \ge 4$, and $k$ an even integer satisfying $4 \le k \le n$. We consider the smooth embedded hypersurface in $\mathbb{R}^{n+1}$ given by the equation
\begin{equation} \label{eq:CLa}
   (r - a)^2 + (x^n)^2 + (x^{n+1})^2 = 1, 
\end{equation}
where $a>1$ is a constant to be determined.  This hypersurface is obtained by rotating a $2$-dimensional unit sphere about a $2$-dimensional coordinate plane, and is therefore homeomorphic to $\mathbb{S}^2 \times \mathbb{S}^{n-2}$. 

We would like to prove that, for any $a \in ( 1+ b(n,k) , n/2)$ the hypersurface defined by \eqref{eq:CLa} satisfies that $\sigma_k(A) \ge 0$ and $\sigma_1(A)$ changes signs, where
\[
   b(n,k) = \frac{n-k}{k-1} + \frac{1}{k-1} \sqrt{\frac{(n-1)(n-k)}{k}} \ge 0.
\]
That the interval $( 1 + b(n,k) , n/2)$ is nonempty for $n \ge k \ge 4$ is justified by Proposition~\ref{pr:a} below. (However, the interval is empty for $k =2$ and all $n \ge 2$.)

We first derive a formula for $\sigma_k(A)$ for all $n \ge 3$ and $1 \le k \le n$. Let
\[
   \psi(r, x^n) = - \sqrt{1 - (r -a )^2 - (x^n)^2}, \qquad 0 < a -1 \le r \le a + 1.
\]
By rotation symmetry, it suffices to carry out the calculation at a point where $x^1 = \cdots = x^{n-2} = 0$ and $x^{n-1} = r$. Unless otherwise indicated, we let Greek letters such as $\alpha, \beta$ range from $1$ to $n$, and English letters such as $i, j$ range from $1$ to $n-1$. We denote $\psi_{\alpha}= \p \psi/\p x^{\alpha}$ and $\psi_{\alpha\beta} = \p^2 \psi/\p x^{\alpha} \p x^{\beta}$.

Note that the induced metric is given by
\[
   g_{\alpha \beta} = \delta_{\alpha \beta} + \psi_{\alpha} \psi_{\beta}.
\]
It follows that the matrix, at the point $(0, \dots, 0, r, x^n)$, 
\[
   (g_{\alpha \beta}) = \begin{bmatrix}
   I_{n-2} & 0 \\
   0 & T_2
   \end{bmatrix}.
\]
Here $I_m$ denotes the $m \times m$ identity matrix, and $T_2$ is a $2 \times 2$ matrix defined by
\begin{align}
   T_2 & = \left[ \begin{matrix}
         1 + (\psi_{r})^2 & \psi_r \psi_n \\
         \psi_r \psi_n & 1 + (\psi_n)^2
   \end{matrix}
   \right] \notag \\
   & = (-\psi)^{-2} \left[ \begin{matrix}
         1 - (x^n)^2 & x^n (r - a) \\
         x^n (r - a) & 1 - (r - a)^2
   \end{matrix}
   \right], \label{eq:T2}
\end{align} 
in which
\[
  \psi_r = \frac{\partial \psi}{ \partial r} = \frac{r - a}{- \psi}, \qquad \psi_n = \frac{\partial \psi}{ \partial x^n} = \frac{x^n}{- \psi}.
\]
On the other hand, the second fundamental form
\[
   A_{\alpha \beta} = \frac{\psi_{\alpha \beta}}{\sqrt{1 + |D \psi|^2}} = (-\psi) \psi_{\alpha \beta}.
\]
Then, at the point $(0,\ldots, 0, r, x^n)$, the second fundamental form matrix 
\[
   (A_{\alpha \beta}) = \left[ \begin{matrix}
   (1 - a/r) I_{n-2} & 0 \\
   0 & T_2
   \end{matrix}
   \right],
\]
where $T_2$ is the $2 \times 2$ matrix given by \eqref{eq:T2}. Hence, the matrix of the shape operator is given by
\begin{align*}
   (A^{\beta}_{\alpha})
   = (A_{\alpha \gamma}) (g_{\beta \gamma})^{-1} 
    = \left[ \begin{matrix}
   (1 - a/r) I_{n-2} & 0 \\
   0 & I_2
   \end{matrix}
   \right]. 
\end{align*}
Therefore, we have
\begin{equation}\label{eq:H}
   H = \sigma_1(A) = (n-2)t + 2,
\end{equation}
and for any $2 \le k \le n$, 
\begin{equation} \label{eq:eg2k}
  \sigma_k(A) = \binom{n - 2}{k} t^k + 2 \binom{n-2}{k-1} t^{k-1} + \binom{n-2}{k-2} t^{k-2}.
\end{equation}
Here we denote $t = 1 - a/r$, which satisfies that
\begin{align} \label{eq:ranget}
    - \frac{1}{a - 1} <  t < \frac{1}{a + 1}, \qquad \mbox{for all $a - 1 < r < a + 1$}.
\end{align}
Moreover, in \eqref{eq:eg2k}, we use the combinatoric convention so that 
  \begin{equation*} 
     \sigma_n (A) = t^{n-2}, \qquad \sigma_{n-1}(A) = 2t^{n-2} + (n-2) t^{n-3}.
  \end{equation*}
  
Now return to our setting $n \ge k \ge 4$ and $k$ being even.
Clearly, for $k = n$ (thus $n$ is even), we always have $\sigma_n(A) \ge 0$. 
For $k < n$ and $k$ being even, $\sigma_k(A) \ge 0$ if 
\begin{equation} \label{eq:quad}
		\binom{n - 2}{k} t^2 + 2 \binom{n-2}{k-1} t + \binom{n-2}{k-2} \ge 0.
\end{equation}
Note that the inequality \eqref{eq:quad} holds for all $t \ge t_1$, where 
\[
   t_1 = - \frac{k-1}{n - k} \left[\sqrt{\frac{n - 1}{k(n-k)}} + 1\right]^{-1} = - \frac{1}{b(n,k)}.
\]
On the other hand, by \eqref{eq:H} we have $H < 0$ for
\[
		t < - \frac{2}{n - 2}.
\]
Thus, if we can show that $t_1 < - 2/(n-2)$, i.e.,
\begin{equation} \label{eq:nt1}
   \frac{n}{2} > 1 - t_1^{-1} 
   = 1 + b(n,k),
\end{equation}
then by \eqref{eq:ranget} for any real number $a$ satisfying
\begin{equation} \label{eq:defa}
   - \frac{2}{n-2} > - \frac{1}{a-1} > t_1, \qquad \textup{namely, \quad $\frac{n}{2} > a > 1 + b(n,k)$},
\end{equation}
we have that $\sigma_k(A) \ge 0$ and that $\sigma_1(A)$ changes signs; more precisely, for a fixed $a \in (1 + b(n,k),n/2)$, 
\begin{align*}
   \sigma_1(A) & > 0, \qquad \textup{for $\frac{n-2}{n} a < r < a + 1$, \quad and} \\
   \sigma_1(A) & < 0, \qquad \textup{for $a - 1 < r < \frac{n-2}{n}a $}.
\end{align*}
Now notice that \eqref{eq:nt1} is assured by Proposition~\ref{pr:a} below. This finishes the proof. 
We remark that, when $k<n$, the same result holds if $a = 1 + b(n,k)$.
\end{ex}

\begin{prop} \label{pr:a}
   For each $n \ge k \ge 4$, 
   \[
       \frac{n}{2} >  1 + b(n,k) = 1 + \frac{n-k}{k-1} + \frac{1}{k-1} \sqrt{\frac{(n-1)(n-k)}{k}}.
   \]
\end{prop}
\begin{proof}
Observe that
  \begin{equation*} 
     \frac{n}{2} - 1 - b(n,k) = c(n,k) \left\{\Big[\frac{(k-3)^2}{4} - \frac{1}{k}\Big] n + \big(k - 2 + \frac{1}{k} \big) \right\} > 0,
  \end{equation*}
  in which
  \[
     c(n,k) = \frac{n}{k-1} \left( \frac{(k-3)n}{2} + 1 + \sqrt{\frac{(n-1)(n-k)}{k}}\right)^{-1} > 0,
  \]
  for $n \ge k \ge 4$.
\end{proof}

\begin{remark}
By varying the parameter $a$ in Example~\ref{ex:odd} and Example~\ref{ex:even}, we can also provide examples of closed hypersurfaces in $\mathbb{R}^{n+1}$, which satisfy that $\sigma_k(A) \ge 0$ but $\sigma_{k-1}(A)$ changes signs, for each $3 \le k \le n$. This in particular answers a question raised by H. D. Cao. It should be also possible to vary $a$ so that $\sigma_k(A) \ge 0$ and $\sigma_l(A)$ changes signs for some $l$ with $ 1\le l < k$ and $k\ge 3$.
\end{remark}

\section{Positive mass theorem for hypersurfaces} \label{se:thm2}

Throughout this section, we denote by $M$ a complete non-compact,  embedded, and orientable $C^{n+1}$-smooth  hypersurface in $\mathbb{R}^{n+1}$, unless otherwise indicated. Let $A$ and $H$ be, respectively, the second fundamental form and the mean curvature of $M$. Recall that 
\begin{eqnarray*}
		M_0 &=& \{ p \in M: A=0 \mbox{ at } p\}.
\end{eqnarray*}
We adopt the convention that $H = - \mbox{div}_0 \nu$, where $\nu$ is a smooth unit normal vector field to $M$ and $\dive_0$ is the Euclidean divergence operator.  For a function $f(x) = f(x^1, \dots, x^n)$, we denote  $f_i = \p f/\p x^i$, $f_{ij} = \p^2 f/\p x^i \p x^j$ for all $i,j= 1,\ldots,n$, and $Df = (f_1,\ldots,f_n)$.

\begin{defi} \label{de:waf}
We say that $M \subset \mathbb{R}^{n+1}$ is an asymptotically flat hypersurface if  $M$ satisfies the following conditions:
\begin{itemize}
\item[(1)] There is a compact subset $K \subset M$ so that $M\setminus K$ consists of countably many components $N_i$, where each $N_i$ is the graph of a function $f_{(i)}$ over the exterior of a bounded region in some hyperplane $\Pi_i$;
\item[(2)] If $\{x^1, \ldots, x^n\}$ are coordinates in $\Pi_i$, we require $\lim_{|x| \rightarrow \infty} f_{(i)}(x) = a_i$, where $a_i$ is either a bounded constant, $a_i = \infty$, or $a_i = -\infty$, and $\lim_{|x| \rightarrow \infty} |Df_{(i)}(x)|=0 $ for each $i$.
\end{itemize}
We refer $N_i$ the ends of $M$. We say that an end $N_i$ is asymptotic to the hyperplane $\Pi_i$ if $a_i=0$. By translation, whenever $|a_i| < \infty$, $N_i$ is asymptotic to a hyperplane.
\end{defi}

\begin{ex}[{\cite{Bray}, \cite[Proposition 2.6]{HW-Penrose}}] \label{ex:schwarzschild}
The spacelike $n$-dimensional ($n\ge 3$) Schwarzschild metric is a complete and conformally flat metric
\[
	\left(\mathbb{R}^n \setminus \{ 0\},  \left( 1+ \frac{m}{2|x|^{n-2}}\right)^{4/(n-2) } \delta \right).
\]
An $n$-dimensional Schwarzschild manifold of $m > 0$ can be isometrically embedded into $\mathbb{R}^{n+1}$, as a spherically symmetric  $C^{\infty}$ asymptotically flat hypersurface of two ends, and each end is the graph of $h(x)$ over $\mathbb{R}^n \setminus B_{(2m)^{1/(n-2)}}$, where
\begin{align*}
	h(x) &= C_0 \pm \sqrt{ 8m (|x| - 2m)}  &&\mbox{if } n =3,\\
	h(x) &= C_0 \pm \sqrt{2m} \ln ( |x| + \sqrt{|x|^2 - 2m} ) &&\mbox{if } n =4,\\
	h(x) &= C_0 \pm O(|x|^{2-\frac{n}{2}}) \quad \mbox{for $|x| \gg 1$} &&\mbox{if } n \ge 5,
\end{align*}
for some constant $C_0$.
\end{ex}

In Section \ref{se:thm1}, we have proved that a \emph{closed} hypersurface with nonnegative scalar curvature is weakly mean convex (up to an orientation). Below, we generalize the result to complete asymptotically flat hypersurfaces.

\begin{pf-complete}
 Suppose to the contrary that $H$ changes signs through $\Gamma$. By  Theorem~\ref{th:unbounded}, $\Gamma$ is unbounded. Hence $\Gamma$ must intersect at least one end $N$. 
Lemma \ref{le:geodesic} yields that  $\Gamma$ lies in a hyperplane, say $\{ x^{n+1} = 0\}$, which is tangent to $N$ at the unbounded subset $\Gamma$ of $N$. Because $M$ is asymptotically flat, $N$ is the graph of $f$ over the exterior region of a hyperplane. By the assumption that $|Df| = o(1)$, the unit normal vector to $N$ must converge to $\p/\p x^{n+1}$. Therefore, we can conclude that the end $N$  is asymptotic to $\{ x^{n+1} = 0 \}$. 
Denote by $h = x^{n+1}\big|_M$ the height function on $M$. By Morse--Sard theorem, the level set $h^{-1}(\epsilon)$ is a $C^{n+1}$  submanifold for almost every $\epsilon$. 

Let $\nu$ be the unit normal vector field on $M$ which is pointing upward on $N$, i.e.,
\begin{equation} \label{eq:afnu}
	\nu = \frac{(-Df, 1)}{\sqrt{1+ |Df|^2}} \quad \mbox{at } (x, f(x))\in N.
\end{equation}
Let $H$ be the mean curvature with respect to $\nu$. By Proposition~\ref{pr:nonempty}, for $\epsilon > 0$ sufficiently small, $h^{-1}(\epsilon) \cap \{ p\in M: H > 0 \mbox{ at } p\} \ne \emptyset$. Note that, for almost every $\epsilon >0$, the mean curvature of each connected component of $h^{-1}(\epsilon)$ that intersects  $\{ p\in M: H > 0 \mbox{ at } p\}$ is non-negative,  because $H$ can only change signs through an unbounded subset by Theorem~\ref{th:unbounded} and $M$ has only countably many ends.

Notice that $h^{-1}(\epsilon)$ is closed for almost every $0< \epsilon \ll 1$. Let $\Sigma_{\epsilon}$ be the \emph{outermost} connected component of $h^{-1}(\epsilon)$ such that $\Sigma_{\epsilon} \cap \{ p\in M: H > 0 \mbox{ at } p\} \ne \emptyset$, i.e., it is not enclosed by any other connected component of $h^{-1}(\epsilon)$ which intersects $\{ p\in M: H > 0 \mbox{ at } p\}$. 

Now we fix a sufficiently small $\epsilon >0$ so that $\Sigma_{\epsilon}$ has nonempty intersection with $N$ and $|\nabla^M h|\neq 0$ on every point in $\Sigma_{\epsilon}$. 
Because $f$ tends to zero at infinity and $\Sigma_{\epsilon}$ is outermost, $\eta =  Df/ |Df|$ is the \emph{inward} unit normal vector on $\Sigma_{\epsilon} \cap N$.  Note $\langle \nu, \eta\rangle < 0$ on $\Sigma_{\epsilon} \cap N$. Because $|\nabla^M h | \neq 0$ on every point in $\Sigma_{\epsilon}$,  $\langle \nu, \eta \rangle$ is strictly negative everywhere on  $\Sigma_{\epsilon}$. Denote by $H_{\Sigma_{\epsilon}}$ the mean curvature with respect to $\eta$.  Apply Corollary~\ref{co:HHR} to obtain $H_{\Sigma_{\epsilon}} \le 0$. This contradicts the compactness of $\Sigma_{\epsilon}$ (for,  a compact set has at least one convex point at which $H_{\Sigma_{\epsilon} }>0$). Therefore, $H$ has a sign on $M$.
\qed
\end{pf-complete}

\begin{cor} \label{co:half-space}
Let $M$ be a complete connected $C^{n+1}$  asymptotically flat hypersurface of countably many ends in $\mathbb{R}^{n+1}$ with nonnegative scalar curvature. Suppose that an end $N$ of $M$ is asymptotic to the hyperplane $\Pi$. Then $N$ strictly lies in one side of $\Pi$, unless $M$ is identical to $\Pi$.
\end{cor}
\begin{pf}
Without loss of generality, assume that $\Pi=\{ x^{n+1} = 0\}$. Suppose $N$ is the graph of a function $f$ over $\{x^{n+1}=0\}\setminus B_{r_1}$ for some $r_1>0$ and $ |f(x)| \to 0 $ as $|x| \to \infty$. We would like to prove that $N$ is contained  in either $\{x^{n+1} > 0\}$ or $\{ x^{n+1} < 0\}$, unless $M$ is identical to $\{ x^{n+1} = 0\}$. 

By Theorem~\ref{th:halfspace}, the mean curvature of $M$ has a sign. Suppose $H \ge 0$ with respect to $\nu$, where $\nu$ is the upward pointing unit normal vector on $N$ given by \eqref{eq:afnu}. (Otherwise, we reflect $M$ about $\{ x^{n+1} = 0\}$.) By Lemma~\ref{le:maximum-principle}, 
\[
	\max_{\overline{B}_{r_2} \setminus B_{r_1}} f = \max_{\p B_{r_2}} f \qquad \mbox{for all } r_2 > r_1 .
\]
Because $\max_{\p B_{r_2}} f  \to 0$ as $r_2\to \infty$, we conclude that $f \le 0$ outside $ B_{r_1}$. Moreover, by applying the strong maximum principle to $H(f)\ge 0$, we have $f<0$ outside $B_{r_1}$, unless $f \equiv 0$. In the latter case, we can further conclude that $M$ is identical to $\{ x^{n+1} = 0\}$ by repeating the argument over $B_{r_2 } \setminus B_{r_0}$ for $0\le r_0 < r_1$. 


\qed
\end{pf}

Note that the scalar curvature of a graph has a divergence form~\cite{Reilly:1973} (see also \cite{L}). 
\begin{prop} If $\Omega$ is an open subset in $\mathbb{R}^n$. Let $f\in C^{2}(\Omega)$. Then the scalar curvature of the graph of $f$ is 
  \begin{align}
    R(Df, D^2 f) =  \sum_j \partial_j \sum_i \left(\frac{f_{ii} f_j - f_{ij} f_i}{1+|Df|^2} \right). \label{eq:div}
  \end{align}
\end{prop}
\begin{pf}
By Gauss equation,
\[
     R = 2 \sigma_2 (A) = \sum_{i,j} (A_i^i A_j^j - A_i^j A_j^i),
\]
  in which $A = (A_j^i)$ is the shape operator with
\[
     A_j^i = \partial_j \big(\frac{f_i}{w}\big) = \Big(\frac{f_i}{w}\Big)_j, \qquad \textup{where $w = \sqrt{1 + |D f|^2}$.}
\]
  Observe that
  \begin{align*}
     A_i^i A_j^j 
     & = \partial_j \big[\Big(\frac{f_i}{w} \Big)_i \frac{f_j}{w} \big] - \frac{f_j}{w} \Big(\frac{f_i}{w}\Big)_{ij}, \\
     A_i^j A_j^i & =  \partial_i \big[\Big(\frac{f_i}{w} \Big)_j \frac{f_j}{w} \big] - \frac{f_j}{w} \Big(\frac{f_i}{w}\Big)_{ij}.
  \end{align*}
  It follows that
  \begin{align*}
    R 
    & = \sum_{j} \partial_j \sum_i \left[ \Big(\frac{f_i}{w} \Big)_i \frac{f_j}{w} -  \Big(\frac{f_j}{w} \Big)_i \frac{f_i}{w}\right] \\
    & = \sum_j \partial_j \sum_i \left(\frac{f_{ii} f_j - f_{ij} f_i}{w^2} \right).
  \end{align*}
\qed
\end{pf}

\begin{defi}[cf. \cite{L}] \label{de:mass}
Let $M$ be a $C^2$  asymptotically flat hypersurface. Let $N$ be one end of $M$, which is the graph of $f$ over the exterior of a bounded region in the hyperplane $\Pi$.  The mass of $N$ is  defined by
\begin{align} \label{eq:mass} 
	m &= \frac{1}{2(n-1)  \omega_{n-1}}  \lim_{r\rightarrow \infty}\int_{S_r} \frac{1}{ 1 + |Df|^2 } \sum_{i,j}  (f_{ii} f_j - f_{ij} f_i)  \frac{x^j}{|x|}  \, d\sigma,
\end{align}
where  $S_r = \{ (x^1, \dots, x^n) \in \Pi: |x|=r\}$, $d\sigma$ is the standard spherical volume measure of $S_r$, and $\omega_{n-1}$ is the volume of the unit $(n-1)$ sphere in Euclidean space. 
\end{defi}

\begin{lemma} \label{le:mass-level-set}
Let $M$ be a  $C^2$  asymptotically flat hypersurface.  Let $N$ be an end of $M$, which is the graph of $f$ over the exterior of a bounded region in the hyperplane $\Pi$.  If there exists a bounded region $\Omega_r$ in $\Pi$ such that $\p \Omega_r$ is the disjoint union of  $S_r $ and $\Sigma = \{ x\in \Pi: f(x) = c\}$ for some constant $c$, and that $|Df|$ does not vanish on $\Sigma$, then
\[
	m= \frac{1}{2(n-1)  \omega_{n-1}}\left( \int_{\Sigma} \frac{|Df|^2}{ 1 + |Df|^2} H_{\Sigma}\, d\sigma + \lim_{r \to \infty} \int_{\Omega_r} R(Df, D^2 f) \, dx \right),
\]
where $R(Df, D^2 f)$ is the scalar curvature of the graph of $f$,  $\eta$ is the unit normal vector on $\Sigma$ pointing away from $\Omega_r$, and $H_{\Sigma}$ is the mean curvature of $\Sigma$ with respect to $\eta$.
\end{lemma}
\begin{pf}
Applying the divergence theorem to \eqref{eq:div} over $\Omega_r$ yields
\begin{align*} 
	&\int_{S_r} \frac{1}{ 1 + |Df|^2 } \sum_{i,j=1}^n  (f_{ii} f_j - f_{ij} f_i)  \frac{x^j}{|x|}  \, d\sigma \\
	&=  \int_{\Omega_r} R(Df, D^2 f) \, dx- \int_{\Sigma} \frac{1}{ 1 + |Df|^2 } \sum_{i,j=1}^n  (f_{ii} f_j - f_{ij} f_i) \eta^j  \, d\sigma.
\end{align*}
Because $\Sigma$ is a level set of $f$, $\eta$ equals either $Df/ |Df|$ or $- Df/ |Df|$. If $\eta = - Df/ |Df|$,
\begin{align*}
 	 H_{\Sigma} &= - \mbox{div}_0 \eta= \sum_{i=1}^n \frac{\partial}{ \partial x^i}\left( \frac{f_i}{ |Df|}\right)\notag\\
	& = \frac{ 1}{ |Df|^3} \sum_{i,j=1}^n( f_{ii}f_{j}f_j-  f_{ij} f_i f_j ).
\end{align*}
We then derive
\begin{align} \label{eq:pmt}
	\begin{split}
		&\int_{S_r} \frac{1}{ 1 + |Df|^2 } \sum_{i,j=1}^n  (f_{ii} f_j - f_{ij} f_i)  \frac{x^j}{|x|}  \, d\sigma  \\
		&= \int_{\Omega_r} R(Df, D^2 f) \, dx+ \int_{\Sigma} \frac{|Df|^2}{ 1 + |Df|^2} H_{\Sigma}\, d\sigma.
	\end{split}
\end{align}
If $\eta = Df/ |Df|$, we also derive the same identity. Letting $r\to \infty$, we prove the lemma.
\qed
\end{pf}

Generally,  $\Omega_r$  may not exist. We shall prove that if $M$ has nonnegative scalar curvature, then such $\Omega_r$ exists, and moreover  $H_{\Sigma} \ge 0$.  

\begin{pf2}
We assume that $M$ is not a hyperplane; otherwise the theorem trivially holds. Consider an end $N$ of $M$, and suppose  $N$ is the graph of $f$ over $\{x^{n+1} = 0\} \setminus B_{r_1}$ for some $r_1>0$. By Theorem~\ref{th:halfspace}, $H$ has a sign on $M$. We may without loss of generality assume that $H\ge 0$ with respect to $\nu$, where $\nu$ is the upward unit normal to $N$, given by 
\[
		\nu = \frac{(-Df, 1)}{\sqrt{1+ |Df|^2}}.
\] 
(Otherwise, we may replace $f$ by $-f$.) We divide the proof  into the following cases. \vspace{10pt}

\noindent{\bf Case 1:} $\lim_{|x| \to \infty} f(x) = a$ for some bounded constant $a$. By translation, we may assume that $|f(x)| \to 0$ as $|x| \to \infty$; namely, $N$ is asymptotic to the hyperplane $\{ x^{n+1} = 0 \}$. By proof of Corollary \ref{co:half-space}, $N \subset \{ x^{n+1} < 0\}$. Therefore, for $\epsilon>0$ sufficiently small, some connected components of the level set $\{ x\in \{ x^{n+1} = 0\}: f(x) = -\epsilon\}$ lie in $N$ with no  boundary. We define $\Sigma_{-\epsilon}$ to be an outermost connected component, i.e., $\Sigma_{-\epsilon}$ is not enclosed by other components. By Morse--Sard theorem, $\Sigma_{-\epsilon}$ is  $C^{n+1}$ for almost every $\epsilon$.  Moreover, because $f$ tends to zero,  for some small $\epsilon > 0$,  $\eta = - Df/ |Df|$ is the unit vector on $\Sigma_{-\epsilon}$, pointing inward to the bounded region in $\{x^{n+1} = 0\}$ enclosed by $\Sigma_{-\epsilon}$. Let $H_{\Sigma_{-\epsilon}}$ be the mean curvature of $\Sigma_{-\epsilon}$ defined by $\eta$. Then, $H_{\Sigma_{-\epsilon}} \ge 0$ by Corollary \ref{co:HHR} and by $H\ge 0$.  Applying Lemma~\ref{le:mass-level-set}, we have $m\ge 0$. 

If $m=0$, then $M$ must be identical to $\{ x^{n+1} = 0\}$. For, otherwise, there exists some positive $\epsilon$ so that $\Sigma_{-\epsilon}$ has $H_{\Sigma_{-\epsilon}} \equiv 0$ by \eqref{eq:pmt}. This contradicts compactness of $\Sigma_{-\epsilon}$. \vspace{10pt}

\noindent{\bf Case 2:}  $\lim_{|x| \to \infty} f(x) = \infty$. The set $\{ x\in \{ x^{n+1} = 0 \}: f(x) = \Lambda\}$ lies in $\{ x^{n+1} \} \setminus B_{r_1}$ for $\Lambda \gg 1$. Let $\Sigma_{\Lambda}$ be the outmost component of the above set. For $\Lambda$ sufficiently large, $\eta = -Df/ |Df|$ is the normal vector to $\Sigma_{\Lambda}$ pointing inward to the bounded region enclosed by $\Sigma_{\Lambda}$. Let $H_{\Lambda}$ be the mean curvature with respect to $\eta$. Hence, $H_{\Sigma_{\Lambda}} \ge 0$ by Corollary \ref{co:HHR} and by $H\ge 0$.  Applying Lemma~\ref{le:mass-level-set}, we have $m\ge 0$. If $m=0$, we can show that $M$ is identical to a hyperplane as in Case 1.\vspace{10pt}

\noindent{\bf Case 3:} $\lim_{|x| \to \infty} f(x) = -\infty$. This case cannot happen. Otherwise, for some $\Lambda\gg 1 $, there is a closed submanifold $\Sigma_{-\Lambda} \subset \{ x\in \{x^{n+1}=0\} \setminus B_{r_1} : f(x) =  - \Lambda\}$ so that the unit normal vector $\eta = Df/|Df|$ is pointing inward to the region enclosed by $\Sigma_{-\Lambda}$. Let $H_{\Sigma_{-\Lambda}}$ be the mean curvature with respect to $\eta$. Then, $H_{\Sigma_{-\Lambda}} \le 0$ by Corollary~\ref{co:HHR}. This contradicts compactness of $\Sigma_{-\Lambda}$. 

\qed
\end{pf2}

Last, we verify below that our definition of the mass~\eqref{eq:mass} coincides with the classical definition of the ADM mass, if we assume stronger fall-off rates on the derivatives of $f$. Let us recall the definition of the ADM mass (see, for example, \cite[Equation (4.1)]{Bartnik}). 
\begin{defi}
We say that an $n$-dimensional manifold $(M, g)$ has an asymptotically flat end $N$ if $N\subset M$ is diffeomorphic to $\mathbb{R}^n \setminus B_1$ and $N$ has a coordinate chart $\{y\}$ so that $g_{ij} (y)= \delta_{ij} + O_2(|y|^{-q})$  for some $q> (n-2)/2$. The $O_2$ indicates that first and second derivatives also decay at rates one and two orders faster, respectively. 

For $n\ge 3$, the ADM mass of the asymptotically flat end $N$ is defined by
\begin{align}\label{eq:classical-mass}
	 \frac{1}{2(n-1)  \omega_{n-1}} \lim_{r\rightarrow \infty} \int_{|y|=r} \sum_{i,j} \left(\frac{\partial g_{ij}}{\partial y^i} - \frac{\partial g_{ii} }{\partial y^j} \right) \tau^j\, d\sigma_g,
\end{align}
where $\tau$ is the outward unit normal to $\{ |y|=r\}$ with respect to $g$, $d\sigma_g$ is the volume measure of $\{ |y| = r\}$ with respect to $g$, and $\omega_{n-1} = \textup{vol} (\mathbb{S}^{n-1})$. 
\end{defi}

\begin{lemma}  [cf. \cite{L}] \label{le:mass}
Let $n\ge 2$ and $M$  an $n$-dimensional $C^3$ asymptotically flat  hypersurface.  Let $N$ be an end of $M$ which is the graph of $f$. Let $R$ be the scalar curvature of $M$ and $R\in L^1(N)$. Then, the mass of $N$ defined by \eqref{eq:mass} is  finite.

If in addition $n\ge 3$, $|Df(x)|^2 = O_2(|x|^{-q})$ for some $q> (n-2)/2$, and $|Df(x)|^2 |D^2 f(x)| = o(|x|^{1-n})$ as $|x| \to \infty$, then \eqref{eq:mass} equals \eqref{eq:classical-mass}.
\end{lemma}
\begin{pf}
Applying the divergence theorem to \eqref{eq:div} yields
\begin{align*}
	 & 2(n-1)  \omega_{n-1}  m\\
	 & =  \int_{S_{r_0}} \frac{1}{ 1 + |Df|^2 } \sum_{i,j =1}^n  (f_{ii} f_j - f_{ij} f_i)  \frac{x^j}{|x|}  \, d\sigma+ \lim_{r\rightarrow \infty} \int_{r_0 \le |x| \le r} R(Df, D^2 f) \, dx.
\end{align*}
Because $R$ is integrable over $N$ and $|Df(x)|=o(1)$ as $|x| \to \infty$, the second term on the right hand side is bounded.
Therefore, $m$ is bounded.

To prove the second statement, we consider the coordinate chart $\{y\}$ of $N$, where
\[
	y^i =(0, \dots, \underbrace{x^i}_{i-\mbox{th}}, \dots, 0, f(0, \dots, \underbrace{x^i}_{i-\mbox{th}}, \dots, 0)).
\]
Then
\[
	\frac{\partial}{\partial y^i} = \partial_i + f_i \partial_{n+1}, 
\]
where we denote $\partial_i = \frac{\partial}{\partial x^i}$. Moreover, at the point $(x, f(x)) \in N$, 
\begin{align*}
	g_{ij} &= \langle\frac{\partial}{\partial y^i}, \frac{\partial}{\partial y^j}\rangle = \delta_{ij} + f_i f_j,\\
	\frac{\partial g_{ij} }{\partial y^k} &= \frac{\partial (f_i f_j) }{\partial x^k} = f_{ik} f_j + f_i f_{jk}.
\end{align*}
Therefore, $N$ is an asymptotically flat end of $M$ by hypothesis of $|Df|$. Denote by $\mu =\sum_{i=1}^n \frac{x^i}{r} \p_i$ the outer unit normal to $S_r$ in $\{x^{n+1} = 0\}$. Let $\tau$ be  the outer unit normal to the graph of $f$ over $S_r$ in $M$. Then, 
\[	
	\tau = \frac{ \mu + \mu(f) \partial_{n+1} }{\sqrt{1 +  |\mu(f) |^2}},
\]
and hence, 
\[
	\tau^j = g(\tau,\frac{\partial}{\partial y^j}  ) = \langle\tau,\frac{\partial}{\partial y^j}  \rangle= \frac{\mu^j + \mu (f) f_i}{ \sqrt{1+ | \mu(f)|^2}}.
\]
It follows that
\begin{align}\label{eq:Sr}
\begin{split}
	 & \int_{f( S_r)} \sum_{i,j} (\frac{\partial g_{ij}}{\partial y^i} - \frac{\partial g_{ii} }{\partial y^j} ) \tau^j\, d\sigma_g \\
	 &= \int_{S_r} \sum_{i,j}  (f_{ii} f_j - f_{ij} f_i)  \frac{\mu^j + \mu (f)  f_j}{ \sqrt{1+ |\mu(f)|^2}} \sqrt{1+ |D^T f|^2}\, d\sigma, 
\end{split}
\end{align}
where $d\sigma_g$ and $d\sigma$ are the $(n-1)$-Hausdorff measures on $f(S_r)$ and $S_r$ respectively, and $D^T f$ is the derivative along the directions tangent to $S_r$.  By \cite[Proposition 4.1]{Bartnik} and hypotheses on the derivatives of $f$, the left hand side of \eqref{eq:Sr} converges to \eqref{eq:classical-mass} as $r\to\infty$, and the right hand side of \eqref{eq:Sr} converges to \eqref{eq:mass}. 
\qed
\end{pf}

\appendix
\section{Topological results} \label{se:appendix}

For a topological space $X$, we denote by $\tilde{H}_k(X)$ the $k$th reduced homology group of $X$ with coefficient in $\mathbb{Z}$. Recall that the rank of $\tilde{H}_0(X)$ plus one equals the number of path-connected components of $X$ (see~\cite[p. 110]{Hat}, for example). 
\begin{lemma} \label{le:connected}
Let $X$ be a contractible topological space. Let $U, V$ be two subsets in $X$ so that $X = \textup{int} (U) \cup \textup{int} (V)$. Then
\[
	\tilde{H}_0 (U\cap V) \approx \tilde{H}_0(U) \oplus \tilde{H}_0(V).
\]
 where $\approx$ stands for the group isomorphism. 
 
 As a consequence,  the number of path-connected components of $U$ plus the number of path-connected component of $V$ equals the number of path-connected components of $U\cap V$ plus one.
\end{lemma}
\begin{proof}
 Applying the Mayer--Vietoris sequence to $X = \textup{int}(U) \cup \textup{int}(V)$ yields
    \[
      \cdots \to \tilde{H}_1(X) \to \tilde{H}_0(U\cap V) \to \tilde{H}_0(U) \oplus \tilde{H}_0(V) \to \tilde{H}_0(X) \to 0.
  \]
Because $X$ is contractible and $\tilde{H}_k(X) = 0$ for all $k$, it completes the proof.
\end{proof}

Recall the definition that $E$ \emph{locally separates} $X$ if there exists an open neighborhood $N$ of $E$ so that $N\setminus E$ is disconnected. We say that $E$ \emph{separates} $X$ if $X\setminus E$ is disconnected.
\begin{prop} \label{pr:enclose}
Let $X$ be a contractible topological space. Let $E$ be a connected closed subset of $X$. If $E$ locally separates $X$, then $E$ separates $X$. In particular, if $E$ is bounded, then $E$ encloses a bounded open set $\Omega$ of $X$, i.e. $\Omega \cap E =\emptyset$ and $\partial \Omega \subset E$.
\end{prop}
\begin{proof}
Because $E$ lies in some connected component of $N$, we may without loss of generality assume that  $N$ is connected. Applying Lemma~\ref{le:connected} for $U=X \setminus E$ and $V= N$ yields that $X \setminus E$ has at least two connected components. If $E$ is bounded, $X \setminus E$ has almost one unbounded component. Let $\Omega$ be a bounded component of $X \setminus E$. Then $\partial \Omega \subset E$.
\end{proof}

\begin{prop} \label{pr:boundary}
Let $X$ be a contractible topological space. Let $\Omega$ be a connected open subset of $X$ and $X\setminus \overline{\Omega} = \sqcup U_{\alpha}$ where each $U_{\alpha}$ is a connected component. Then $\partial U_{\alpha}$ is connected.
\end{prop}
\begin{proof}
Without loss of generality, we may assume that $X\setminus \overline{\Omega} = U_{\alpha_0}$ is connected. Otherwise, we may replace $\overline{\Omega}$ by  $\mbox{cl}( \overline{\Omega} \cup (\sqcup_{\alpha \neq \alpha_0} \overline{U}_{\alpha}))$ which is connected, where $\mbox{cl}(E)$ denotes the closure of $E$. Then let $U_{\alpha_0} = X \setminus \mbox{cl}(\overline{\Omega} \cup \sqcup_{\alpha \neq \alpha_0} \overline{U}_{\alpha})$.

Denote $\Sigma = \partial \Omega = \partial U_{\alpha_0}$. Suppose to the contrary that $\Sigma$ is disconnected. Then there exist two disjoint open subsets $A, B\subset X$ with $A\cap \Sigma \neq \emptyset$ and $B\cap \Sigma\neq \emptyset$ so that $\Sigma = \Sigma \cap (A\sqcup B)$. Note that both $A$ and $B$ have non-empty intersection with both  $\Omega$ and $U_{\alpha_0}$, so $\Omega \cup (A\sqcup B)$ and $U_{\alpha_0} \cup (A\sqcup B)$ are both connected.  Applying Lemma~\ref{le:connected} with $U= \Omega \cup (A\sqcup B)$ and $V=U_{\alpha_0} \cup (A\sqcup B)$ yields that $A\sqcup B$ is connected. It leads a contradiction.
\end{proof}
\begin{remark}
The above proposition is not true in general if the condition that $\Omega$ is connected is dropped. For example, let $\Omega$ be the disjoint union of a closed unit ball and a closed annulus in $\mathbb{R}^n$, both centered at the origin. Then its complement contains an annulus which has two boundary components. 
\end{remark}

\begin{thm} \label{th:boundary-separate}
Let $X$ be a contractible topological space.  Let $\Omega \subset X$ be a connected open set and 
  $\overline{\Omega} \ne X$.   Denote $X \setminus \overline{\Omega} =  \sqcup_{\alpha} U_{\alpha}$ where $U_{\alpha}$ are connected components. By Proposition~\ref{pr:boundary}, $\partial U_{\alpha}$ is connected. Let $\Gamma$ be a connected component of  $\partial \Omega$ that contains $\partial U_{\alpha_0}$ for some $\alpha_0$. Then $\Gamma$ separates $X$. In particular, if $\Gamma$ is bounded, then $\Gamma$ encloses a bounded open set $W\subset X$ so that $\partial W\subset\Gamma$.
\end{thm}
\begin{proof}
Note that
  \[
     X\setminus \Gamma = \{U_{\alpha_0}\sqcup (X \setminus \overline{U}_{\alpha_0}) \}\cap (X \setminus \Gamma)
  \]
 and $U_{\alpha_0}$ and $X \setminus \overline{U}_{\alpha_0} $ are disjoint open sets, so $\Gamma$ separates $X$. The rest follows from Proposition~\ref{pr:enclose}.
  \end{proof}
  \begin{remark}\label{re:boundary-component}
  In general, not every connected component of $\partial \Omega$ would separates $X$.  For example, let $\Omega = \mathbb{R}^n \setminus \{ \sqcup_k B_k \cup \{\mbox{the origin}\}\}$ where $B_k$ are disjoint closed balls centered at $(2^{-k},0, \dots, 0)$ of radius $2^{-k-3}$. Then, the origin is a connected component of $\p  \Omega$, but it does not separate $\mathbb{R}^n$.
  \end{remark}

\end{document}